\documentclass{amsart}

\usepackage[english]{babel}
\usepackage{amsmath}
\usepackage{graphicx}
\usepackage{mathrsfs}
\usepackage{color}
\usepackage[T1]{fontenc}
\usepackage{esint}
\theoremstyle{plain}
\newtheorem{theorem}{Theorem}[section]
	
\newtheorem{lemma}[theorem]{Lemma}
\theoremstyle{definition}

\theoremstyle{remark}
\newtheorem{remark}[theorem]{Remark}

\numberwithin{equation}{section}
\newcommand{\AXY}{\mathcal{AXY}_\varepsilon(\Omega)}
\newcommand{\newatop}{\genfrac{}{}{0pt}{1}} 
\newcommand{\co}{\mathrm{Co}\,}
\newcommand{\e}{\varepsilon}
\newcommand{\el}{\mathrm{el}}

\newcommand{\ep}{\varepsilon}
\newcommand{\vep}{\varepsilon}
\newcommand{\vth}{\vartheta}
\newcommand{\vph}{\varphi}
\newcommand{\ffi}{\varphi}
\newcommand{\W}{\mathcal{W}}
\newcommand{\nren}{\mathcal{W}}
\newcommand{\Rect}{\mathrm{Rect}}
\newcommand{\C}{\mathbb{C}}
\newcommand{\R}{\mathbb{R}}
\newcommand{\N}{\mathbb{N}}
\newcommand{\Z}{\mathbb{Z}}

\newcommand{\F}{\mathcal{F}}

\newcommand{\Om}{\Omega}
\newcommand{\D}{\mathcal{D}}
\newcommand{\DD}{{\mathcal{D}}_M^\n}

\newcommand{\lin}{\Gamma}
\newcommand{\Lin}{\Omega_\rho^\lin(\mu)}
\newcommand{\weakly}{\rightharpoonup}
\newcommand{\eff}{\mathrm{sym}}
\newcommand{\weakstar}{\stackrel{*}{\weakly}}
\newcommand{\flt}{\stackrel{\mathrm{flat}}{\rightarrow}}
\newcommand{\fla}{\mathrm{flat}}
\newcommand{\AF}{\mathcal{AF}_{\varepsilon}(\Omega)}
\newcommand{\loc}{\mathrm{loc}}

\newcommand{\negl}{\mathrm{ng}}
\newcommand{\lef}{\mathrm{left}}
\newcommand{\rig}{\mathrm{right}}
\newcommand{\res}{\mathop{\hbox{\vrule height 7pt width .5pt depth 0pt \vrule height .5pt width 6pt depth 0pt}}\nolimits}
\newcommand{\tube}{T^N_{\delta,\rho}}
\newcommand{\tubep}{T^N_{\delta,\rho,\ep}}
\newcommand{\n}{{\bf n}}
\newcommand{\nn}{(\n)}

\newcommand{\Int}{A}
\DeclareMathOperator{\diff}{d}
\DeclareMathOperator{\ud}{d}
\DeclareMathOperator{\diam}{diam}
\DeclareMathOperator{\dist}{dist}
\DeclareMathOperator{\di}{dist}

\DeclareMathOperator{\supp}{supp}

\title[Fractional vortices and string defects] {$\Gamma$-convergence analysis\\ of a generalized $XY$ model:\\ fractional vortices and string defects}

\author[R. Badal]
{R. Badal}
\address[Rufat Badal]{Zentrum Mathematik - M7, Technische Universit\"at  M\"unchen, Boltzmannstrasse 3, 85748 Garching, Germany}
\email[R. Badal]{badal@ma.tum.de}
\author[M. Cicalese]
{M. Cicalese}
\address[Marco Cicalese]{Zentrum Mathematik - M7, Technische Universit\"at  M\"unchen, Boltzmannstrasse 3, 85748 Garching, Germany}
\email[M. Cicalese]{cicalese@ma.tum.de}
\author[L. De Luca]
{L. De Luca}
\address[Lucia De Luca]{Zentrum Mathematik - M7, Technische Universit\"at  M\"unchen, Boltzmannstrasse 3, 85748 Garching, Germany}
\email[L. De Luca]{deluca@ma.tum.de}
\author[M. Ponsiglione]
{M. Ponsiglione}
\address[Marcello Ponsiglione]{Dipartimento di Matematica ``Guido Castelnuovo'', Sapienza Universit\`a di Roma, P.le Aldo Moro 5, I-00185 Roma, Italy}
\email[M. Ponsiglione]{ponsigli@mat.uniroma1.it}

\begin{document}

\maketitle
\begin{abstract}
We propose and analyze a generalized two dimensional $XY$ model, whose  interaction potential has  $n$ weighted  wells, describing corresponding symmetries of the system.   As the lattice spacing vanishes, we derive by $\Gamma$-convergence the  discrete-to-continuum limit of this model. 
In the energy regime we deal with, the asymptotic ground states exhibit   fractional vortices, connected by string defects. 
The $\Gamma$-limit takes into account both contributions, through a renormalized energy, depending on the configuration of fractional vortices, and a surface energy, proportional to the length of the strings. 

Our model describes in a simple way  several topological singularities arising  in Physics and Materials Science. Among them, 
disclinations and string defects in liquid crystals,  fractional vortices and domain walls in  micromagnetics,  partial dislocations and stacking faults in crystal plasticity.

\end{abstract}
\vspace{.5cm}


\noindent{\bf Mathematical Classification:} 82D30, 82B20, 49J45 

\tableofcontents

\section*{Introduction}
Since the pioneering paper by Kosterlitz and Thouless \cite{KT}, the $XY$ model is considered the classical example of discrete spin system exhibiting phase transitions mediated by the formation and the interaction of topological singularities. 
Even at zero temperature the model presents interesting features: Depending on the energy regime, the geometry of the ground states is very rich, going from uniform to disordered states, all the way through isolated vortex singularities and clustered dipoles. Both topology and energy concentration take place at different length scales, thus making the analysis fascinating and popular also in the mathematical community.



In this paper we focus on the variational analysis of a two dimensional modified $XY$ model at zero temperature, describing the formation of fractional vortices and string defects. Our main motivation comes from observing that this kind of singularities characterizes several discrete systems in Physics and Materials Science such as disclinations and string defects in liquid crystals \cite{PMC}, fractional vortices and domain walls in micromagnetics \cite{TC},  partial dislocations and stacking faults in crystal plasticity \cite{HB}. We do not focus on the specific details of any of such models, rather we aim at providing a simple variational model highlighting some of their relevant common features. 

Given a bounded open set $\Omega\subset \R^2$, the classical $XY$ energy on the lattice $\e\Z^{2}\cap\Omega$ is given by
\begin{equation}
XY_{\e}(v)=-\sum_{\langle i,j\rangle}(v(i),v(j)),
\end{equation}
where the sum is taken over pairs of nearest neighbors $\langle i,j\rangle$; i.e., $i,j\in\e\Z^{2}\cap\Omega$ with $|i-j|=\e$, and $v:\Omega\cap\ep\Z^2\to\mathcal S^1$ is the spin field. By writing $v=e^{i\vth}$, the $XY$ energy can be rewritten - up to constants - as
\begin{equation}\label{intro-XY}
XY_{\e}(v)=\frac 1 2\sum_{\langle i,j\rangle}f(\vth(j)-\vth(i)),
\end{equation}
where $f(t)=1-\cos(t)$. 
The rigorous upscaling as $\ep\to 0$ of the $XY_\e$ functional - and more in general of discrete spin systems governed by $2\pi$-periodic potentials $f$ - has been recently obtained in \cite{AC, ACP, ADGP, DL} in terms of $\Gamma$-convergence. Loosely speaking, configurations with energy of order  $C |\log\ep|$  exhibit  a finite number (controlled by the pre-factor $C$) of vortex-like singularities.
Around each singularity the order parameter $v$ looks like a fixed rotation of the map $\big(\frac{x}{|x|}\big)^{d}$ where $d\in\Z$ is the degree of the singularity. There the energy concentrates and blows up as $|d||\log \ep|$. This analysis  is then refined exploiting the next lower order term in the energy expansion. Indeed, after removing the logarithmic leading order contribution, a finite interaction energy, referred to as {\it renormalized energy}, remains. The renormalized energy depends on the positions and the degrees of the vortices, and it is considered the main driving force responsible for their dynamics \cite{ADGP,ADGP2}. The strategy adopted to analyze the $XY$ model exploits methods and tools from the earlier analysis developed for the continuous Ginzburg-Landau functionals \cite{BBH, JS, ABO, SS, AP}. Moreover, these techniques have been successfully used to understand in terms of $\Gamma$-convergence the well-known analogy between vortices in the $XY$ model and in superconductivity, and screw dislocations in anti-plane linearized elasticity \cite{ACP, P}.

To some extent, this analogy extends {\it mutatis mutandis} to many other topological defects in  liquid crystals, micromagnetics and crystal plasticity. Let us introduce and motivate our model using the language of planar uniaxial nematics. In the Oseen-Frank and Ericksen picture, nematics are described by a vector field $u=e^{i\ffi}$, the {\it director}, that (in two dimensions) takes values in $\mathcal S^1$. The molecules are supposed to interact via an elastic potential depending only on their orientation $\ffi$. If the nematic molecules present
 {\it head-to-tail symmetry}, the interaction potential is $\pi$-periodic. Therefore there is a clear identification between a $2$-dimensional nematic lattice model and the $XY$ spin model. Doubling the phase variable $\ffi$ of the director, the interaction potential becomes $2\pi$-periodic so that the nematic energy can be written as in \eqref{intro-XY} with $\vth=2\ffi$. This corresponds to the Lebwohl and Lasher functional (\cite{LL}) recently studied by $\Gamma$-convergence in \cite{BCS}. Since in the classical $XY$ model integer vortices appear, as result of the identification $\vth=2\ffi$, for liquid crystals the relevant topological singularities are half vortices, namely the disclinations.

The presence of half vortices is related to orientability issues of the director, that can be better described and visualized in a continuous framework: 
in a configuration consisting of two half vortices, the director field cannot be smooth everywhere. Indeed, it has antipodal orientations on the two sides of a line joining the half vortices. According to \cite{LG}, we call this kind of discontinuity lines {\em string defects} (see Figure \ref{fig: dipoles}). A first rigorous study of orientability issues related to the presence of half vortices can be found in \cite{BZ2}, while
the right space to describe discontinuous directors has been very recently identified in \cite{Bed} as a suitable subspace of $SBV$, the space of {\it special functions with bounded variation}.

\begin{figure}[http!]
\begin{center}
\includegraphics[scale=0.166 ]{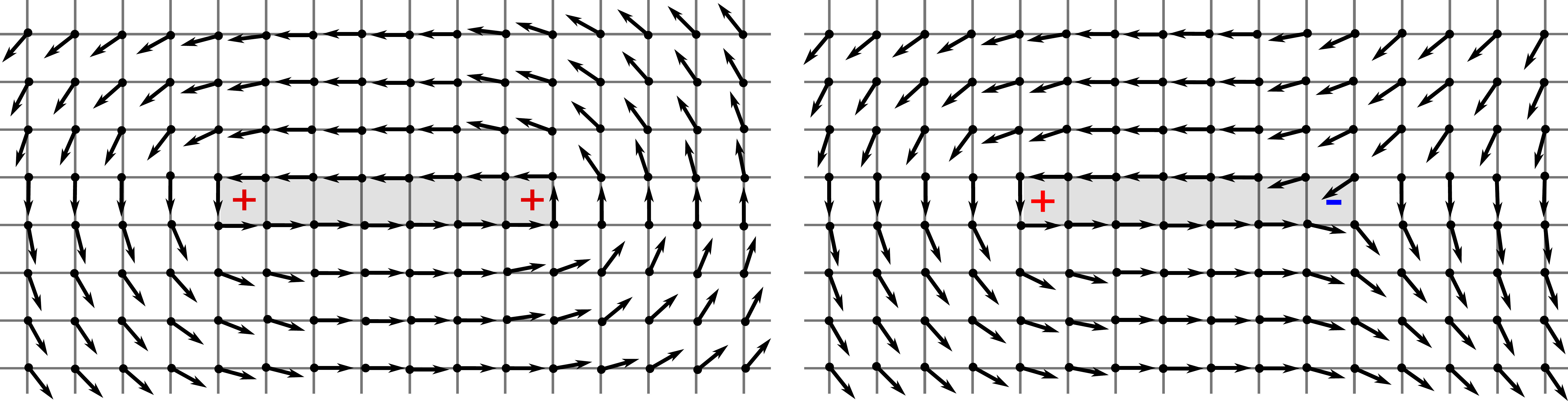}
\caption{The spin field close to two half vortex pairs (the degree of the vortices is $(+1/2,+1/2)$ in the left picture and $(+1/2,-1/2)$ in right one). The shaded regions highlight necessary antipodal spins forming the string defects.}\label{fig: dipoles}
\end{center}
\begin{picture}(0,0)
\end{picture}
\end{figure}
In this paper we provide a quantitative  analysis of such orientability issues, by introducing energy functionals which describe 
 both half, and in fact fractional vortices,  and string defects. To this end, we propose a modified $XY$ model, that draws back to \cite{CC, K, LG, PMC, BW}. We consider $2\pi$-periodic potentials, acting on nearest neighbors, and having $\n\in\N$ wells in $[0,2\pi)$; in one well the potential is zero, while in the remaining $\n -1$ it is positive but vanishing as $\ep$ goes to zero. In this way, we model systems with one symmetry, and $\n -1$ vanishing  asymmetries. The case $\n=2$ provides a simple toy model for  liquid crystals with small head-to-tail asymmetry, like those made of non centrosymmetric or chiral molecules \cite{LT}.

More specifically,  we consider  potentials $f_{\e}^{(\n)}$ with $\n -1$ wells of order $\ep$ (see Figure \ref{pot}). 
\begin{figure}[h!]
	\centering
	\def\svgwidth{0.9\textwidth}
	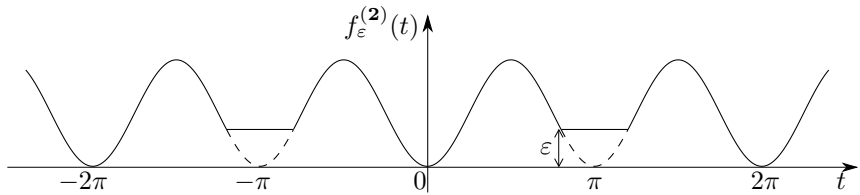
	\caption{ The plot of the potential $f_\ep^{\nn}$ for $\n=2$ obtained truncating $1-\cos(2t)$ at level $\e$ around the odd wells.}	\label{pot}
\end{figure}
Having $\n$ wells in the potential produces  fractional vortices of degree $\pm \frac{1}{\n}$, while   
the additional energy  due to the weighted $\n -1$ wells yields, together with the renormalized energy, a new term in the limit, depending on the length of the string defects. 
The specific $\e$-scaling of the wells makes finite the string defect energy, and hence of the same order of the renormalized energy.
Other scalings could be considered, leading to different limit theories.  

Let us describe our model and results in more details. Given a discrete director field configuration $u_{\e}=e^{i\ffi_{\e}}$, we introduce the modified $XY$ energy functional 

\begin{equation}\label{intro: energy}
	\F^{\nn}_\vep(\vph) := \frac 1 2\sum_{\langle i,j\rangle} f^{\nn}_\vep(\vph(j)-\vph(i)).
\end{equation}
We show that configurations $u_{\e}=e^{i\ffi_{\e}}$ with
$
\F^{\nn}_\vep(\vph_{\e})\simeq M\pi|\log\e|
$
are consistent with the formation of $M$ fractional vortices at $x_{1},\dots,x_{M}$ having degrees $\pm \frac{1}{\n}$. Moreover,  up to a subsequence, $u_{\e}$ strongly converges in $L^{1}(\Omega;\R^2)$ to some unitary field $u$ belonging to $SBV(\Omega;\mathcal S^1)$ with $u^{\n}\in H^{1}_{\loc}(\Omega\setminus\{x_{1}, \dots, x_{M}\} ; \mathcal S^1)$. Indeed in Theorem \ref{mainthm} we prove that the functionals $\F^{\nn}_\vep(\vph_{\e})-M\pi|\log\e|$ $\Gamma$-converge to the limit energy $\F^{(\n)}_0(u)$ given by
\begin{equation}\label{Gammalim0}
\F^{(\n)}_0(u):= \W(u^\n)+M\gamma+\int_{S_u}|\nu_u|_1\,{\ud\mathcal H}^{1}.
\end{equation}
Here $\W(u^\n)$ represents the renormalized energy; it depends on the positions and signs of the fractional vortices, but also on the excess energy  due to unneeded oscillations of $u$; indeed, minimizing $\W(u^\n)$ with respect to all director fields $u$ compatible with the given configuration of fractional vortices gives back the classical 
renormalized energy for vortices in superconductivity \cite{BBH}. The constant $\gamma$ is a core energy stored around each singularity, and has memory of the discrete lattice structure. The third addendum is the anisotropic length of the string, reminiscent of the lattice symmetries. We note that the analysis by $\Gamma$-convergence of discrete functionals leading to such anisotropic interfacial energies is a well established research field \cite{ABC,BC,CDL}. On the other hand, limit energies as in \eqref{Gammalim0} may appear also as $\Gamma$-limits of  continuous Ginzburg-Landau type functionals. Actually, in \cite{GMM} the authors propose and study a variant of the standard Ginzburg-Landau energy,  that could be seen as a continuous counterpart of our modified $XY$ model. 

We remark that our $\Gamma$-convergence result can be considered the first  step towards the study of the joint dynamics of strings and vortices, as well as partial dislocations and stacking faults. 
We expect that the competition between the surface and the renormalized energies generates non-trivial metastable configurations. As a matter of fact, neglecting boundary effects, two equally charged vortices -  in the classical $XY$ model - repel each other with a force proportional to the inverse of their distance, so that they are never in equilibrium.
On the other hand, since in the generalized $XY$ model the surface term diverges with the distance between the vortices, there exists a critical length at which the two forces balance and equilibrium is reached.  In this respect, it seems interesting to investigate the dynamics of strings and fractional vortices driven by the limit energy functional $\F_0^{(\n)}$,  through  both rigorous mathematical analysis  and  numerical simulations.



\section{Preliminaries on $BV$ and $SBV$ functions}

Here we list some preliminaries on $BV$ and $SBV$ functions that will be useful in the following. We begin by recalling some standard notations. 

Let $A\subset\R^2$ be open bounded  with Lipschitz boundary. As customary,  $BV(A;\R^d)$ (resp. $SBV(A;\R^d)$) denotes the set of functions of bounded variation (resp. special functions of bounded variation) defined on $A$ and taking values in $\R^d$.  Moreover, we set $BV(A):=BV(A;\R)$ and $SBV(A):=SBV(A;\R)$. Given any set $\mathcal I\subset \R^d$, the classes of functions $BV(A;\mathcal I)$ and $SBV(A;\mathcal I)$ are defined in the obvious way.
We refer to the book \cite{AFP} for the definition and the main properties of $BV$ and $SBV$ functions.

Here we recall that the distributional gradient $\mathrm{D} g$ of a function $g\in SBV(A;\R^d)$ can be decomposed as:
\[
	\mathrm{D} g=\nabla g\,\mathcal L^2\res A+(g^+ - g^-)\otimes \nu_g \,\mathcal{H}^{1}\res{S_g},
\]
where $\nabla g$ is the approximate gradient of $g$, $S_g$ is the jump set of $g$, $\nu_g$  is a unit normal to $S_g$ and $g^{\pm}$ are the approximate trace values of $g$ on  $S_g$.

For $p > 1$, $SBV^p(A;\R^d)$ denotes the subspace of $SBV(A;\R^d)$ defined by
$$
SBV^p(A;\R^d):=\{g\in SBV(A;\R^d) \, : \, \nabla g \in L^p(A;\R^{2d}), \, \mathcal H^{1}(S_g) <+\infty\}.
$$
Moreover, $SBV^p_\loc(A;\R^d)$ denotes the class of functions belonging to $SBV^p(A';\R^d)$ for all $A'\subset\subset A$ open.


We recall three results about compactness, lower semicontinuity and density properties of $SBV^p$ and $SBV^p_\loc$ functions. 

\begin{theorem}[Compactness \cite{A}]\label{SBVComp}
	Let $\{g_h\}\subset SBV^p(A;\R^d)$ for some $p > 1$. Assume that there exists $C>0$ such that
	\begin{equation}\label{compactnessCondition}
		\int_{A} |\nabla g_h|^p \diff x + \mathcal{H}^{1}(S_{g_h}) + \left\| g_h	\right\|_{L^\infty (A;\R^d)}\leq C\quad\textrm{for all }h\in\N.
	\end{equation}
	Then, there exists $g \in SBV^p(A; \R^d)$ such that, up to a subsequence,
\begin{equation}\label{sbvconv}
\begin{split}
	&g_h \to g \textrm{ (strongly) in } L^1(A;\R^d),\\
	&\nabla g_h \weakly \nabla g \textrm{ (weakly) in } L^p(A;\R^{2d}),\\
	&\liminf_{h\to\infty} \mathcal H^1(S_{g_h}\cap A') \geq \mathcal H^1(S_g \cap A'),
\end{split}
\end{equation}
for any open set $A'\subseteq A$.
\end{theorem}

In the following, we say that a sequence $\{g_h\} \subset SBV^p(A;\R^d)$ weakly converges  in $SBV^p(A;\R^d)$ to a function $g\in SBV^p(A;\R^d)$, and we write that $g_h\weakly g$ in $SBV^p(A;\R^d)$,  if $g_h$ satisfy \eqref{compactnessCondition} and $g_h\to g $ in $L^1(A;\R^d)$.
Moreover,  we say that $\{g_h\} \subset SBV_\loc^p(A;\R^d)$ weakly converges  in $SBV_\loc^p(A;\R^d)$ to a function $g\in SBV_\loc^p(A;\R^d)$, and we write that $g_h\weakly g$ in $SBV_\loc^p(A;\R^d)$, if  $g_h\weakly g$ in $SBV^p(A';\R^d)$ for any open set $A'\subset\subset A$.

\begin{theorem}[Lower semicontinuity \cite{AB,AFP}]\label{lscSBV}
Let $K\subset\R^d$ be compact and let $C_{0}>0$. Let $\Theta\in C (K\times K; [C_{0},\infty))$ be a positive, symmetric function satisfying 
\begin{equation}\label{triangineq}
\Theta(a,c)\le\Theta(a,b)+\Theta(b,c)\qquad\textrm{ for all }a,b,c\in K,
\end{equation}
and let $\Psi\in C(\R^d;\R^+)$ be an  even, convex and positively 1-homogeneous function satisfying $\Psi(\nu)\ge C_{0}|\nu|$ for all $\nu\in\mathcal{S}^1$. 

If $\{g_h\}\subset SBV^p_\loc(A;K)$ and $g_h\weakly g$ in $SBV^p_\loc(A;\R^d)$ for some $p>1$, then
\begin{equation}\label{lscclaim}
	\int_{S_g} \Theta(g^+,g^-)\Psi(\nu_g) \diff\mathcal{H}^{1} \leq \liminf_{h \to \infty} \int_{S_{g_{h}}} \Theta(g_{h}^+,g_{h}^-)\Psi(\nu_{g_{h}}) \diff\mathcal{H}^{1}.
\end{equation}
\end{theorem}

Finally, we state the density result  proved in \cite[ Theorem 2.1 and Corollary 2.4]{BCG}.
First we recall that a polyhedral set in $\Omega$ is (up to a $\mathcal H^1$ negligible set) a finite union of segments contained in $\Omega$.

\begin{theorem}[Density \cite{BCG}]
\label{polyapprox}
Let $\mathcal Z\subset\R^d$ be finite and let  $g \in SBV(A; \mathcal Z)$. Then there exists a sequence $\{g_h\} \subset SBV(A;\mathcal Z)$ such that $S_{g_h}$ is polyhedral, $g_h \weakly g$ in $SBV^p(A; \R^d)$ for any $p>1$, and 
$$
	\lim_{h \to \infty} \int_{S_{g_h}} \phi(g_h^+, g_h^-,\nu_{g_h}) \ud \mathcal H^{1} = \int_{S_g} \phi( g^+, g^-,\nu_g) \ud \mathcal H^{1}
$$
for any continuous function $\phi \colon \mathcal Z \times \mathcal Z\times\mathcal S^1 \to \R_+$ satisfying $\phi( a, b,\nu) = \phi( b, a,-\nu)$ for all $( a, b,\nu) \in \mathcal Z \times \mathcal Z\times\mathcal S^1$.
In particular, $|D g_h|(A)\to |D g|(A)$.
\end{theorem}

\section{Description of the problem}
Let $\n\in\N$ be a fixed natural number,  and
let $\Omega\subset\R^2$ be an open bounded set with Lipschitz continuous boundary, representing the domain of definition of the relevant fields in the models we deal with. 
Additionally, to simplify matter, we assume that $\Omega$ is simply connected. This is convenient  in our constructions where we make use of  Poincar\'e Lemma (as in proof of Lemma \ref{bulklemma}). Nevertheless, this assumption can be removed with some slight additional effort in our proofs, by introducing a finite number of ``cuts'' $\gamma_i$ in $\Omega$, such that $\Omega\setminus\cup_i\gamma_i$ is simply connected.

\subsection{The discrete lattice}

For every $\vep>0$, we set 
\[
	\Omega_\vep := \bigcup_{i\in\vep\Z^2 \colon i+\vep Q \subset \overline\Omega}(i+\vep Q), 
\]
where $Q=[0,1]^2$ is the unit square.
Moreover we set $\Omega_{\vep}^0 := \vep\Z^2 \cap \Omega_\vep$,
and $\Omega_{\vep}^1 :=
\left\{(i,j)\in\Omega_{\vep}^0\times\Omega_{\vep}^0:|i-j|=\vep,\, [i,j]\subset\Omega_\ep \right\}$, where $[i,j]$ denotes the (closed) segment joining $i$ and $j$\,.
These objects represent the reference lattice and the class of nearest neighbors, respectively. The cells contained in $\Omega_{\vep}$ are labeled by the set of indices
$\Omega_{\vep}^2 := \left\{i\in\Omega_{\vep}^0 \colon i+\vep Q\subset\Omega_{\vep}\right\}$.
Finally, we define the discrete boundary of $\Omega$ as
$
	\partial_\vep\Omega := \partial\Omega_{\vep}\cap\vep\Z^2.
$
In the following, we will extend the use of such notations to any given  subset $A$ of $\R^2$.

\subsection{Discrete functions and discrete topological singularities}

Here we introduce the classes of discrete functions on $\Omega_\vep^0$ and a notion of discrete topological singularity.
We first set
\[
	\AF := \left\{\psi \colon \Omega_{\vep}^0\to\R\right\},
\]
and we introduce the class of admissible fields from $\Omega_\ep^0$ to the set $\mathcal S^1$ of unit vectors in $\R^2$
\[
	\AXY := \left\{w \colon \Omega_{\vep}^0\to\mathcal S^1\right\}.
\]

For any $\psi\in \AF$, $w\in \AXY$, and  for any $(i,j)\in\Omega_\ep^1$ we set
$$
\ud\psi(i,j):=\psi(j)-\psi(i),\qquad\ud w(i,j):=w(j)-w(i).
$$

Following the formalism in \cite{ADGP}, we introduce a notion of discrete vorticity for any given function $\psi\in\AF$.
To this purpose, let $P:\R\to\Z$ be defined as follows
\begin{equation}\label{defdiP}
	P(t)=\mathrm{argmin}\left\{|t-s| \colon s\in 2\pi \Z\right\},
\end{equation}
with the convention that, if the argmin is not unique, then $P(t)$ is the smallest one.
Let $\psi\in\AF$ be fixed. For every $(i,j)\in\Omega_\ep^1$,
 we define
 the signed distance of the discrete gradient of $\psi$ from $2\pi\Z$, as
\begin{equation}\label{elgrad}
\ud^\el\psi(i,j):=\left\{\begin{array}{ll}
\ud\psi(i,j)-P(\ud\psi(i,j))&\textrm{if }i\le j\\
\ud\psi(i,j)+P(\ud\psi(j,i))&\textrm{if }j\le i,
\end{array}\right.
\end{equation}
where $i\leq j$ means that $i_l\leq j_l$ for $l\in\left\{1,2\right\}$. Notice that by definition 
$\ud^\el\psi(i,j) = - \ud^\el\psi(j,i)$ for all $(i,j)\in \Omega_\ep^1$.  
 For every $i\in \Omega_\vep^2$ we define the discrete vorticity of the cell $i+\ep Q$ as
\begin{multline}\label{disctop}
\alpha_\psi(i) := \frac{1}{2\pi}\left(\ud^\el\psi(i,i+\vep e_1)+\ud^\el\psi(i+\vep e_1,i+\ep e_1+\vep e_2)\right.\\
\left.+\ud^\el\psi(i+\vep e_1+\vep e_2,i+\ep e_2) +\ud^\el\psi(i+\vep e_2,i)\right).
\end{multline}
One can easily see that the vorticity $\alpha_\psi$ takes values in $\{-1,0,1\}$.

Stokes' Theorem in our discrete setting reads as follows.
Let $A\subset\Omega$  with $A_\ep$  bounded and simply connected.
Set $L:=\sharp \partial_\ep A$ and let $\partial_\ep A:=\{i^1,\ldots,i^L\}$, with $(i^l,i^{l+1})\in A_\ep^1$ for any $l=1,\ldots,L-1$, and notice that $(i^L,i^1)\in A^1_\ep$. Then, for any $\psi\in\AF$, it holds
\begin{equation}\label{dacitare}\sum_{l=1}^{L-1}\ud^\el\psi(i^l,i^{l+1})+\ud^\el\psi(i^L,i^1)=
2\pi\sum_{i\in A_\ep^2}\alpha_\psi(i).
\end{equation}

We define the vorticity measure $\mu(\psi)$ as follows
\begin{equation}\label{defvor}
	\mu(\psi) :=\pi \sum_{i\in\Omega_\vep^2}\alpha_\psi(i) \delta_{i+ \frac{\vep}{2} (e_1 + e_2)}.
\end{equation}

%
%
Let $A\subset\R^2$.
For any $\mu=\pi\sum_{i=1}^Nd_i\delta_{x_i}$ with $N\in\N$, $d_i\in\Z\setminus\{0\}$ and $x_i\in A$,  the {\it flat norm} of $\mu$ is defined as
$$
\|\mu\|_{\fla}:=\sup_{\|\eta\|_{W_0^{1,\infty}(A)}\le 1}\langle\mu,\eta\rangle.
$$
Whenever it will be convenient, we will  declare  the domain $A$ of the test functions $\eta$ in the definition of the flat norm by writing $\|\mu\|_{\fla(A)}$ instead of $\|\mu\|_{\fla}$.
We will denote by $\mu_n \flt \mu$ the flat convergence of $\mu_n$ to $\mu$.


It is well-known that the flat norm of $\mu$  is related to integer $1$-currents $T$ with $\partial T=\mu$ and having minimal mass (for the theory and terminology of integer currents we refer the reader to \cite{F}). More precisely, 
by \cite[Section 4.1.12]{F} (see also \cite[formula (6.1)]{ADGP}) there holds
$$
\min_{\partial T=\mu}|T|=\sup_{\eta\in \mathrm{Lip}^1(A)}\langle\mu,\eta\rangle,
$$
where $\mathrm{Lip}^1(A)$ is the set of functions $\eta$ with $\sup_{\newatop{x,y\in A}{y\neq x}}\frac{|\eta(y)-\eta(x)|}{|y-x|}\le 1$ .
Since for every $\eta\in W_0^{1,\infty}(A)$ we have
$$
\sup_{\newatop{x,y\in A}{y\neq x}}\frac{|\eta(y)-\eta(x)|}{|y-x|}\le\|\eta\|_{W^{1,\infty}_0(A)}\le (1+\diam(\Omega))\sup_{\newatop{x,y\in A}{y\neq x}}\frac{|\eta(y)-\eta(x)|}{|y-x|},
$$
we deduce that
\begin{equation}\label{minimalconn}
\min_{\partial T=\mu}|T|\ge \|\mu\|_{\fla}\ge \frac{1}{1+\diam(\Omega)}\min_{\partial T=\mu}|T|.
\end{equation}

\subsection{The discrete energy} \label{discreteEnergies}

Here we will introduce a class of energy functionals defined on $\mathcal{AF}_{\vep}(\Omega)$. Let  $f \colon \R\to\R^+$ be a continuous ${2\pi}$-periodic function such that

\begin{itemize}
	\item[(i)] $f(t)=0$ if and only if $t\in {2\pi}\Z$;
	\item[(ii)] $f(t)\geq1-\cos t$ for any $t\in\R$;
	\item[(iii)] $f(t)=\frac{t^2}{2}+O(t^3)$ as $t\to 0$.
\end{itemize}
and set $f^{\nn}(t):=f(\n\,t )$. For any fixed $\vep>0$, we consider the  pairwise interaction potentials $f^{\nn}_\vep:\R\to\R^+$ defined by
\begin{equation}\label{potential}
f_\ep^{\nn}(t):= f^{\nn}(t)\vee \vep\sum_{k\in\Z}\chi_{[(\frac \pi \n+2k\pi,-\frac \pi \n+(2k+2)\pi]}(t).
\end{equation}
We stress  that the  analysis done in this paper might be extended with minor changes to a larger class of potentials such as the smooth ones considered in \cite{K}.

For any $\vph\in\mathcal{AF}_{\vep}(\Omega)$, we define
\begin{equation}\label{energy}
	\F^{\nn}_\vep(\vph) := \frac 1 2\sum_{(i,j)\in\Omega_{\vep}^1} f^{\nn}_\vep(\vph(j)-\vph(i)).
\end{equation}

For the convenience of the reader, we now introduce two additional energy functionals, which will be useful in the proof of our main result. We define
\begin{eqnarray}\label{effenergy}
	&&\F^{\eff}_\vep(\vth):=\frac 1 2\sum_{(i,j)\in\Omega_{\vep}^1}f(\vth(j)-\vth(i)),\quad\textrm{for any }\vth\in\mathcal{AF}_{\vep}(\Omega),\\ 
\label{XY}
	&&XY_\vep(v):=\frac 1 4\sum_{(i,j)\in\Omega_{\vep}^1}|v(j)-v(i)|^2,\quad\textrm{for any } v\in\mathcal{AXY}_\ep(\Omega).
\end{eqnarray}
Let  $\vph\in \AF$ and set $\vth_\ffi:=\n\,\ffi$, $u_\ffi:=e^{i \ffi}$ and $v_\ffi:=e^{i\vth_\ffi}=u_\ffi^\n$.
By (ii), we have that
\begin{equation}\label{comparxy}
	\F^{\nn}_\vep(\vph)\geq\F^{\eff}_\vep(\vth)\geq \frac 1 2\sum_{(i,j)\in\Omega_{\vep}^1}(1-\cos(\vth_\ffi(j)-\vth_\ffi(i)))=XY_\vep(v_\ffi).
\end{equation}

In the following we will consider also the localized version of the energy $\F^{\nn}_\vep(\vph)$, defined for any set $D\subset\Omega$ by
$$
	\F^{\nn}_\vep(\vph,{D}) :=\frac 1 2 \sum_{(i,j)\in {D}_{\vep}^1} f^{\nn}_\vep(\vph(j)-\vph(i)).
$$
The localized versions of the energies  $\F_\ep^\eff(\cdot)$ and $XY_\ep(\cdot)$ on a set $D\subset\Omega$ are analogously defined and are denoted by  $\F_\ep^\eff(\cdot,D)$ and  $XY_\ep(\cdot,D)$, respectively.

\subsection{Extensions of the discrete functions}\label{extdiscrfun}
Here we introduce the extensions of our discrete order parameters to the whole domain.
 
In order to extend the discrete functions from $\Omega_\ep^0$ to the whole $\Omega_\ep$, we will use two types of interpolations depending on  whether the function is scalar or vector valued. To this purpose, we consider the triangulation defined by the following sets:
\begin{equation}\label{triang}
	\begin{split}
		&\{T_i^-\}_{i\in\Omega_\vep^2}:=\{\co(i,i+\vep\,e_1,i+\vep\,e_1+\vep\,e_2)\}_{i\in\Omega_\vep^2},\\
		&\{T_i^+\}_{i\in\Omega_\vep^2}:=\{\co(i,i+\vep\,e_1+\vep\,e_2,i+\vep\,e_2)\}_{i\in\Omega_\vep^2},
	\end{split}
\end{equation}
where, for any $i,j,k\in\R^2$, $\co(i,j,k)$ denotes their convex envelope.
Let  
$w\in\mathcal{AXY}_\ep(\Omega)$.
According  to \eqref{triang}, we define the piecewise affine interpolations of $w$ in an as follows:
for any $i=(i_1,i_2)\in\Omega_\ep^2$ and for any $x=(x_1,x_2)\in i+\ep Q$, we set

\begin{equation*}
A(w)(x)\!:=\!\left\{\begin{array}{l}
\!\!\!\! w(i)\!+\!\frac{\ud w(i,i+\ep e_1)}{\ep}(x_1-i_1)\!+\!\frac{\ud w(i+\ep e_1,i+\ep e_1+\ep e_2)}\ep(x_2-i_2),\,\textrm{if }x\in {T}_i^-,\\[1mm]
\!\!\!\! w(i)\!+\!\frac{\ud w(i+\ep e_2,i+\ep e_1+\ep e_2)}{\ep}(x_1-i_1)\!+\!\frac{\ud w(i,i+\ep e_2)}{\ep}(x_2-i_2),\,\textrm{if }x\in {T}_i^+.
\end{array}\right.
\end{equation*}
One can easily check that, for any open set $D\subset\Omega$ and for any $w\in\mathcal{AXY}_\ep(\Omega)$,
\begin{equation}\label{XYDir}
XY_\ep(w,D)\ge\frac 1 2\int_{D_\ep}|\nabla A(w)|^2\ud x.
\end{equation}


Let $\ffi\in\AF$.
We say that $(i,j)\in\Omega_\vep^1$ is a jump pair of $\vph$, if $\mathrm{dist}(\ud\vph(i,j),2\pi\Z)>\frac\pi \n$ and we denote by $JP_\vph$ the  set of jump pairs of $\vph$.
Furthermore we say that $i+\vep\,Q$ ($i\in\Omega_\vep^2$) is a jump cell,
if  $(j,k)$ is a jump pair of $\ffi$ for some bond 
$$
(j,k)\in\{(i,i+\ep e_1),(i+\ep e_1,i+\ep e_1+\ep e_2),(i+\ep e_2,i+\ep e_1+\ep e_2), (i,i+\ep e_2)\}.
$$
 We denote by $JC_{\vph}$ the set of jump cells of $\vph$.

%
Now, recalling that $u_\ffi =e^{i\ffi}$, we set  
\begin{equation}\label{paiofu}
	\hat u_\vph(x) :=
	\begin{cases}
	          u_\vph(i)& \text{if }x \in i+\vep \overset{\circ}{Q} \text{ with } i + \vep Q\in JC_{\vph},\\
		{A(u_\ffi)}&\text{otherwise in }\Omega_\ep.
	\end{cases}
\end{equation}
Finally, recalling that $v_\ffi =e^{i \n \ffi}$, we set $\hat v_\ffi:= A(v_\ffi)$.   
With a little abuse of notations we identify the functions $\hat u_\ffi$ and $\hat v_\ffi$ with $L^1$ functions defined on $\Omega$, just by extending them to $0$ in $\Omega\setminus\Omega_\ep$. 


\section{The main result} 

In this section we state  our $\Gamma$-convergence result for the energies $\F_\vep^{\nn}$ defined in \eqref{energy}. To this purpose, we precisely define its $\Gamma$-limit $\F_0^{\nn}$.
As mentioned in the Introduction \eqref{Gammalim0}, the functional $\F_0^{\nn}$ is given by the sum of  three terms: the {\it renormalized energy} $\W$, representing the energy far from the limit singularities, the core energy $\gamma$, and the anisotropic surface term measuring the length of the string defects.

\subsection{The $\Gamma$-limit}
We start by defining the renormalized energy $\W$.
Fix $M\in\N$, we set
\begin{multline}\label{Dn}
\mathcal{D}_M :=\big\{v\in W^{1,1}(\Omega;\mathcal{S}^1)\,:\,Jv 
=\pi\sum_{i=1}^{M} d_i\delta_{x_i} \text{ for some }  d_{i}\in\{-1,1\}, \, x_i \in\Omega\\
\text{ with }		 x_i\neq x_j \text{ for } i\neq j, \text{ and }
v \in H^1_\loc(\Omega \setminus(\supp Jv);\mathcal{S}^1)\big\},
\end{multline}
where $Jv$ denotes the distributional Jacobian of $v=(v_1,v_2)$ (see for instance \cite{JS}), defined by
$$
Jv=\textrm{div}(v_1\partial_{x_2}v_2,-v_1\partial_{x_1}v_2).
$$

As shown in \cite{ADGP}, if $v\in\mathcal D_M$ with $Jv:=\pi\sum_{i=1}^Md_i\delta_{x_i}$, then the function
\begin{equation}\label{defw}
w(\sigma,\Omega):= \frac 1 2 \int_{\Omega\setminus\bigcup_{i=1}^M B_\sigma(x_i)}|\nabla v|^2 \ud x-M\pi|\log\sigma|
\end{equation}
is monotonically decreasing with respect to $\sigma$. Therefore, it is well defined the functional $\mathcal{W}: \D_M \to \overline{\R}$ given by
\begin{equation}\label{renormalizedEnergy}
	\W(v):= \lim_{\sigma\to 0} w(\sigma,\Omega).
\end{equation}
We introduce also the localized version of the energy $\W$. For any open $D\subseteq\Omega$ with $\supp Jv\subset D$, we consider $w(v,D)$ defined as in \eqref{defw} with $\Omega$ replaced by $D$, and set
\begin{equation}\label{Wlocal}
	\W(v,D):=\lim_{\sigma\to 0} w(\sigma,D).
\end{equation}

\begin{remark}
Set
\begin{multline}\label{newD}
\bar{\mathcal{D}}_M :=\big\{v \in H^1_\loc(\Omega \setminus(\bigcup_{i=1}^M
\{x_i\});\mathcal{S}^1) \text{ for some }  x_i \in\Omega\\
 \text{ with } x_i\neq x_j \text{ for } i\neq j,
\text{ and }\deg(v,\partial B_\sigma(x_i))=\pm 1\\
\textrm{ for any }\sigma<\min_{i\neq j}\{\textstyle \frac 1 2 |x_i-x_j|, \di(x_i,\partial\Omega)\}
\big\},
\end{multline}
where $\deg(v,\partial B_\sigma(x_i))$ denotes the degree of $v$ on $\partial B_\sigma(x_i)$ (see for instance \cite{BN,AP}).
As shown in \cite{ADGP}, the renormalized energy $\W(v)$  in \eqref{renormalizedEnergy} is well defined for any $v\in \bar{\mathcal{D}}_M$. Moreover, by \cite[Remark 4.4]{ADGP},  if $\W(v)< + \infty$, then the Dirichlet energy of $v$ is uniformly bounded on all diadic annuli around each $x_i$.
By using H\"older inequality (on each diadic annulus) we obtain that $v\in W^{1,1}(\Omega;\mathcal S^1)$. In particular the class of functions in $\bar{\mathcal{D}}_M$ with $\W(v)< + \infty$ coincides with the class of functions in ${\mathcal{D}}_M$ with $\W(v)< + \infty$.
Finally, we notice that if $v\in\bar{\mathcal{D}}_M$ with $\W(v)< +\infty$,
integration by parts easily yields
$$
Jv=\pi\sum_{i=1}^M\deg(v,\partial B_\sigma(x_i))\delta_{x_i}.
$$

\end{remark}


In order to define the ``finite core energy''  $\gamma$, we consider an auxiliary minimization problem. Given $0<\vep < \sigma$, we set
\begin{equation}\label{gammaep}
\gamma(\vep,\sigma):= \min_{\vph \in \mathcal{AF}_\vep(B_\sigma)} \left\{\F^{\eff}_\vep(\vth,B_\sigma): \vth(\cdot) = \theta(\cdot) \text{ on } \partial_\vep B_\sigma\right\},
\end{equation}
where $\F_\vep^{\eff}$ is defined in \eqref{effenergy} and  $\theta (x)$ denotes the angular polar coordinate for any $x\in\R^2\setminus\{0\}$.
Moreover, we set
\begin{equation}\label{gamma}
	\gamma:=\lim_{\vep\to 0}\gamma(\vep,\sigma)-\pi\big|\log \frac\vep\sigma\big|.
\end{equation}
By \cite[Theorem 4.1]{ADGP} the above limit exists, is finite and does not depend on $\sigma$.

Finally, for any $M\in\N$, we define the domain of the $\Gamma$-limit as
\begin{equation}\label{barDn}
	\begin{split}
		\mathcal{D}^{\n}_M := &\left\{u\in SBV(\Omega;\mathcal{S}^1)\,: \, u^\n \in\mathcal D_M,\, u \in SBV^2_\loc(\Omega \setminus (\supp Ju^\n);\mathcal S^1),\right.\\
		&\left.\qquad\qquad\qquad\qquad\qquad\qquad\qquad\qquad\qquad \qquad\mathcal H^1(S_u) < +\infty\right\},
	\end{split}
\end{equation}
and for any $u\in \mathcal D_M^\n$ we set
\begin{equation}\label{Gammalim}
\F^{\nn}_0(u):= \W(u^\n)+M\gamma+\int_{S_u}|\nu_u|_1\ud \mathcal{H}^1,
\end{equation}
where, for any $\nu=(\nu_1,\nu_2)\in\R^2$,  $|\nu|_1:=|\nu_1| +|\nu_2|$ denotes the $1$-norm of $\nu$.
\subsection{The $\Gamma$-convergence result}
We are now in a position to state our $\Gamma$-convergence result.
We recall that the assumptions on the interaction potentials $f^{\nn}_\vep$ are listed in Subsection \ref{discreteEnergies}.
Moreover, given $\vph\in\mathcal{AF}_\vep(\Omega)$ we recall that $\vth_\ffi:=\n\,\ffi$ and $v_\ffi:=e^{i\vth_\ffi}$.
In order to ease the notations, for any function $\vph_\vep\in\mathcal{AF}_\vep(\Omega)$,
we will denote the fields  $\vth_{\vph_\vep},\, u_{\vph_\vep},\, \hat u_{\vph_\vep},\,  v_{\vph_\vep}, \, \hat v_{\vph_\vep}$ 
(introduced in Subsection \ref{extdiscrfun}) 
by $\vth_\ep,\, u_\vep,\, \hat u_\ep,\,  v_\vep, \, \hat v_\vep$, respectively;
for the same reason, we will denote by $\mu_\vep$ the discrete vorticity measure $\mu(\vth_\vep)$.

\begin{theorem}\label{mainthm}
Let $\n,M\in\N$ be fixed. The following $\Gamma$-convergence result holds true.

\medskip
\noindent (i) (Compactness) Let $\{\vph_\vep\} \subset \mathcal{AF}_\varepsilon(\Omega)$ be a sequence satisfying
\begin{equation} \label{energyBound}
	\F^{\nn}_\vep(\vph_\vep) \leq M \pi |\log\vep|+C,
\end{equation}
for some $C\in\R$.
Then, up to a subsequence, 
{$\mu_\vep \flt \mu = \pi \sum_{i = 1}^N d_i \delta_{x_i}$ with  $N\in\N$, $d_i \in \Z \setminus \{0\}$, $x_i \in \Omega$,  $x_i\neq x_j$ for $i\neq j$
 and $\sum_{i=1}^N|d_i|\leq M$. }

Moreover, if $\sum_i|d_i| = M$, then $|d_i| = 1$ and there exists $u \in \DD$ with $Ju^\n=\mu$ such that,
up to a further subsequence, $\hat{u}_\vep \weakly u$  in $SBV_\loc^2(\Omega\setminus \cup_{i=1}^M\{x_i\};\R^2)$. 

\medskip 
\noindent (ii) ($\Gamma$-liminf inequality) Let $u \in\DD$ and
let $\{\vph_\vep\}\subset\mathcal{AF}_\vep(\Omega)$ be such that $\mu_\vep\flt J u^\n$
and $\hat u_\vep\to u$ in $L^1(\Omega;\R^2)$.
Then,
\begin{equation}\label{gammaliminf}
	\liminf_{\vep \to 0} \F^{\nn}_\vep(\vph_\vep) - M \pi |\log\vep| \geq \F^{\nn}_0(u).
\end{equation}

\medskip 
\noindent (iii) ($\Gamma$-limsup inequality) Given $u\in \DD$, there exists $\{\vph_\vep\}\subset\mathcal{AF}_\vep(\Omega)$
such that $\mu_\vep\flt J u^\n$, $\hat u_\vep\weakly u$ in $SBV^2_\loc(\Omega\setminus\bigcup_{i=1}^M\{x_i\};\R^2)$ and
\begin{equation}\label{gammalimsup}
	\lim_{\vep \to 0} \F^{\nn}_\vep(\vph_\vep) - M \pi |\log\vep| = \F^{\nn}_0(u).
\end{equation}
\end{theorem}

A similar $\Gamma$-convergence result for the functional $\F_\ep^\eff$ has been proved in \cite{ADGP}.
For our purposes it is convenient to present here its precise statement using our notations.

\begin{theorem}\label{gammal2}

Let $M\in\N$ be fixed. 
Let $\{\vth_\vep\}\subset\mathcal{AF}_{\vep}(\Omega)$ be such that
$\F^\eff_{\vep}
(\vth_\vep) \leq M\pi|\log\vep|+C$, for some $C\in\R$. 
Then, up to a subsequence, $\mu_\vep\flt\mu=\pi \sum_{i=1}^N d_i\delta_{x_i}$ with $N\in\N$, $d_i\in\Z\setminus\{0\}$, 
$x_i\in\Omega$, $x_i\neq x_j$ for $i\neq j$ and $\sum_{i=1}^{N}|d_i| \leq M$; moreover, there exists a constant  $\bar C\in\R$ such that for any $i=1,\ldots,M$ 
and for any $\sigma<\frac 1 2\dist(x_i,\partial\Omega\cup\bigcup_{j\neq i} x_j)$, there holds
\begin{equation}\label{liminfloc}
	\liminf_{\vep\to 0}\F^\eff_\vep(\vth_\vep,B_{\sigma}(x_i))-\pi|d_i|\log\frac\sigma\vep\geq\bar C.
\end{equation}
In particular,
\begin{equation}\label{liminfomega}
	\liminf_{\vep\to 0}\F^\eff_\vep(\vth_\vep)-\pi|\mu|(\Omega)\log\frac\sigma\vep\geq M\bar C.
\end{equation}
Furthermore, if $\sum_{i=1}^{N}|d_i|=M$, then
\begin{itemize}
	\item[(a)]  $|d_i| = 1$ for any $i=1,\ldots,M$;
	\item[(b)] up to a further subsequence, $\hat v_\vep {\rightharpoonup}v$ in $H^1_{\loc}(\Omega\setminus(\supp \mu);\R^2)$ for some $v \in\D_M$ with $Jv= \mu$;
	\item[(c)] $\displaystyle \liminf_{\vep\to 0}\F^\eff_{\vep}(\vth_\vep,D)-M\pi|\log\vep|\geq \W(v,D)+M\gamma$, for any open $D\subseteq\Omega$ with $\supp Jv\subset D$.
\end{itemize}
\end{theorem}

\begin{remark} 
Since in \cite{ADGP} it is not explicitly stated that $Jv=\mu$, for the reader's convenience we provide here a short proof of this fact.

As proven in \cite{ACP}, if $\F^\eff_{\vep}(\vth_\vep) \leq C|\log\vep|$, then 
\begin{equation}\label{diffjac}
J\hat v_\ep-\mu_\ep\flt 0\quad\textrm{ as }\ep\to 0.
\end{equation}

Now let $\F^\eff_{\vep}(\vth_\vep) \leq M\pi|\log\vep|+C$, $\mu_\ep\flt\mu=\pi \sum_{i=1}^N d_i\delta_{x_i}$ with $|d_i|=1$ and $x_i\in\Omega$, $x_i\neq x_j$ for $i\neq j$, and   $\hat v_\vep {\rightharpoonup}v$ in $H^1_{\loc}(\Omega\setminus(\supp \mu);\R^2)$ for some $v \in\D_M$. 
Let $\mathscr J:=( v_1\partial_{x_2}v_2,-v_1\partial_{x_1}v_2)$ and let $\mathscr{J}_\ep$ be analogously defined with $v$ replaced by $\hat v_\ep$. Then, $Jv=\mathrm{div} \mathscr{J}$, $J\hat v_\ep=\mathrm{div} \mathscr{J}_\ep$, and $\mathscr{J}_\ep\to\mathscr{J}$ in $L^2_\loc(\Omega\setminus\bigcup_{i=1}^M\{x_i\};\R^2)$.

Let $\sigma$ be as in \eqref{newD} and let $\eta:\R_+\to [0,1]$ be the piecewise affine function such that $\eta=1$ in $[0,\frac\sigma 2]$,  $\eta=0$ in $[\sigma,+\infty)$ and $\eta$ is affine in $[\frac{\sigma}{2},\sigma]$. 
Fix $i\in\{1,\ldots,M\}$ and set $\eta_i(x):=\eta(|x-x_i|)$ for any $x\in\Omega$.
Then, in view of \eqref{diffjac}, we have
\begin{multline*}
d_i=\langle \mu,\eta_i\rangle=\lim_{\ep\to 0}\langle \mu_\ep,\eta_i\rangle=\lim_{\ep\to 0} \int_{B_\sigma(x_i)}J\hat v_\ep \eta_i\ud x
\\
=-\lim_{\ep\to 0}\int_{B_\sigma(x_i)}\mathscr{J}_\ep\cdot \nabla\eta_i\ud x
=-\int_{B_\sigma(x_i) \setminus B_{\frac{\sigma}{2}}(x_i)}\mathscr{J} \cdot \nabla\eta_i\ud x 
=\deg(v,\partial B_\sigma(x_i)),
\end{multline*}
which combined with the fact that $Jv=0$ in $\Omega\setminus\bigcup_{i=1}^M\{x_i\}$, yields $Jv=\mu$.

\end{remark}
\section{Proof of Theorem \ref{mainthm}}
This section is devoted to the proof of Theorem \ref{mainthm}. 
In what follows the letter $C$ will denote a constant in $\R$ that may change from line to line;  if the constant $C$ will depend on some parameters $\alpha_1,\ldots,\alpha_k$ we will write $C_{\alpha_1,\ldots,\alpha_k}$.
Moreover, for any $\rho>0$, for any $D\subset\Omega$ open and for any $\nu=\pi\sum_{i=1}^N d_i\delta_{y_i}$, we set
\begin{equation}\label{holedomain}
D_\rho(\nu):=D\setminus\bigcup_{i=1}^N B_\rho(y_i).
\end{equation}

\subsection{Proof of (i): Compactness}

By \eqref{comparxy} and by \eqref{energyBound}, we immediately get
\begin{equation}\label{energyBound2}
\F^\eff_{\vep}(\vth_\vep)\leq M \pi |\log\vep|+C;
\end{equation}
therefore, by Theorem \ref{gammal2}, we have that, up to a subsequence, $\mu_\vep\flt\mu$ for some $\mu$ with all the desired properties.
Let us assume that $\sum_{i=1}^N|d_i|=M$.  By Theorem \ref{gammal2}(a), we get that $|d_i|=1$, and by  Theorem \ref{gammal2}(b), up to passing to a further subsequence, $\hat v_\ep{\rightharpoonup}v$ in $H^1_{\loc}(\Omega\setminus(\supp \mu);\R^2)$ for some $v \in\D_M$ with $Jv=\mu$. 

Now we prove that $\hat{u}_\vep \weakly u$  in $SBV_\loc^2(\Omega\setminus \cup_{i=1}^M\{x_i\};\R^2)$, for some 
 $u \in \DD$ with $u^\n=v$, and hence, in particular, $Ju^\n=\mu$.
Consider  an increasing sequence $\{\Omega^h\}$ of open smooth sets compactly contained in $\Omega$ such that  $\cup_{h\in\N}\Omega^h=\Omega$.
Fix $h\in\N$ (large enough)  and let $\rho>0$ (small enough) be such that the balls $B_{2\rho}(x_i)$ are pairwise disjoint and contained in $\Omega^h$.
We first prove that $\{\hat u_\vep\}$ satisfies the upper bound \eqref{compactnessCondition} in $\Omega^h_\rho(\mu)$, i.e., that
\begin{equation}\label{compCondTrue}
	\sup_{\vep>0}\bigg\{\int_{\Omega^h_{\rho}(\mu)}|\nabla \hat u_\vep|^2\diff x+\mathcal{H}^1(S_{\hat u_\vep}\res \Omega^h_{\rho}(\mu))+\|\hat u_\vep\|_{L^\infty(\Omega^h_{\rho}(\mu);\R^2)}\bigg\}<C_{\rho,h},
\end{equation}
for some constant $C_{\rho,h}$.

Let $\vep>0$ be small enough so that $\Omega^h\subset\Omega_\vep$. 
We preliminarly notice that
\begin{equation}\label{infinitybound}
	|\hat u_\vep|\leq 1\quad \textrm{ in }\Omega^h.
\end{equation}
Moreover, (recalling the notations introduced in Subsection \ref{extdiscrfun}), we set 
\[
J_{\vep}:=\bigcup_{i+\vep Q\in JC_{\vph_\vep}}(i+\vep Q);
\]
it is easy to check, that for any $(i,j)\in\Omega_\vep^1$ with $i,j\notin J_\vep$, it holds
\begin{equation}\label{bounduv}
	\frac{|u_\vep(j)-u_\vep({i})|^2}{|v_\vep(j)-v_\vep(i)|^2}=\frac{1-\cos(\vph_\vep(j)-\vph_\vep(i))}{1-\cos(\n(\vph_\vep(j)-\vph_\vep(i)))}
	\leq C_\n,
\end{equation}
for some $C_\n>0$.
Using that $\hat u_\vep$ is piecewise constant in $J_\vep$, \eqref{comparxy} and \eqref{XYDir},we get
\begin{multline}\label{appgradbound}
	\int_{\Omega^h_\rho(\mu)}|\nabla \hat u_\vep|^2\diff x=\int_{\Omega^h_\rho(\mu)\setminus J_\vep}|\nabla \hat u_\vep|^2\diff x\leq C_\n\int_{\Omega^h_\rho(\mu)}|\nabla \hat v_\vep|^2\diff x\\
	\leq C_\n \, XY_\vep(v_\vep,\Omega_{\frac{\rho}{2}}(\mu))
	\leq C_\n\,\F_\vep^\eff(\vth_\vep, \Omega_{\frac{\rho}{2}}(\mu))\leq C_\n C_{\rho},
\end{multline}
where the last inequality is a consequence of \eqref{energyBound} and \eqref{liminfloc},  and $C_{\rho}>0$.

We now show that
\begin{equation}\label{boundjump}
	\mathcal{H}^1(S_{\hat u_\vep}) \leq C,
\end{equation}
for some constant $C$ independent of $\rho$, $h$ and $\vep$.
First, it is easy to see that 
\begin{equation*}
\mathcal H^1(S_{\hat u_\vep})\leq C\vep\sharp JP_{\vph_\vep}\leq C\F_\vep^\n(\vph_\vep,J_\vep);
\end{equation*}
then, 
in order to prove \eqref{boundjump}, it is enough to prove that
$\F_\vep^{\nn}(\vph_\vep,J_\vep \cap \Omega^h_\rho(\mu))$ is uniformly bounded with
respect to $\vep$, $\rho$ and $h$. By \eqref{energyBound}, \eqref{comparxy} and \eqref{liminfloc}, we have
\begin{multline}\label{lunga}
	C\geq \F_\vep^{\nn}(\vph_\vep)-M\pi|\log\vep|\geq\sum_{i=1}^M(\textstyle \F_\vep^\eff(\vth_\vep, B_\rho(x_i))-\pi\log\frac\rho\vep)\\
	+\F_\vep^{\nn}(\vph_\vep,\Omega^h_\rho(\mu))-M\pi|\log\rho|
	\geq \bar C +\F_\vep^{\nn}(\vph_\vep,\Omega^h_\rho(\mu))-M\pi|\log\rho|
	\\
		 \geq \bar C + \F_\vep^\eff(\vph_\vep,\Omega^h_\rho(\mu)\setminus
  J_\vep)+ \F_\vep^{\nn}(\vph_\vep,\Omega^h_\rho(\mu) \cap J_\vep) -M\pi|\log\rho|
  \\ \geq
  \bar C + \frac{1} 2\int_{\Omega^h_{2\rho}(\mu) }|\nabla \hat v_\vep|^2  \chi_{\Omega^h_{2\rho}(\mu)\setminus J_\vep} \diff x-M\pi|\log\rho|
  + \F_\vep^{\nn}(\vph_\vep,\Omega^h_\rho(\mu) \cap J_\vep)
\end{multline}
Since
\begin{equation}\label{massJvep}
	|J_\vep|=\vep^2\sharp JC_{\vph_\vep}\leq 2\vep^2 \sharp JP_{\vph_\vep}\leq C\ep\F_\ep^{\nn}(\ffi_\ep)\leq C\vep|\log\vep|,
\end{equation}
we have 
$\nabla \hat v_\vep \chi_{\Omega^h_{2\rho}(\mu)\setminus J_\vep} \weakly \nabla v$ in $L^2 (\Omega^h_{2\rho}(\mu);\R^2)$.    
 By lower semicontinuity, we obtain
\begin{multline*}
\frac{1} 2\int_{\Omega^h_{2\rho}(\mu)}|\nabla\hat v_\vep|^2 \chi_{\Omega^h_{2\rho}(\mu)\setminus J_\vep}  \diff x-M\pi|\log\rho|
\geq
\\
\frac{1} 2\int_{\Omega^h_{2\rho}(\mu)}|\nabla v|^2\diff x-M\pi|\log(2\rho)|-M\pi\log 2+r(\ep)
\ge\W(v)-M\pi\log 2+r(\ep,\rho,h),
\end{multline*}
where $\lim_{\ep\to 0}r(\ep)=\lim_{h\to\infty}\lim_{\rho\to 0}\lim_{\ep\to 0}r(\ep,\rho,h)=0$. 
This, together with \eqref{lunga} yields $ \F_\vep^{\nn}(\vph_\vep,\Omega^h_\rho(\mu) \cap J_\vep) \le C+ r(\ep,\rho,h)$ and in turn \eqref{boundjump}.

Therefore, by \eqref{infinitybound},\eqref{appgradbound} and \eqref{boundjump}, the bound \eqref{compCondTrue} is satisfied
and by Theorem \ref{SBVComp} there exists a unitary field $u_{h,\rho}$ such that, up to a subsequence, $\hat u_\vep\weakly u_{h,\rho}$ in $SBV^2(\Omega^h_\rho(\mu);\R^2)$.  By a standard diagonal argument, up to a further subsequence, $\hat u_\ep\weakly u$ in $SBV_\loc^2(\Omega\setminus\bigcup_{i=1}^M\{x_i\};\R^2)$, for some $u\in SBV_\loc^2(\Omega\setminus\bigcup_{i=1}^M\{x_i\};\mathcal S^1)$. Moreover, by \eqref{boundjump}, $\mathcal H^1(S_u)$ is finite.

In order to complete the proof it remains to show that  $u^\n = v$. 
To this purpose, we set $w_\vep:=\hat u_\vep^\n$. Let $i\in\Omega_\ep^2$ with $i+\ep Q\subset\Omega_\vep\setminus J_\vep$ and let $x\in T_i^-$, with $T_i^-$ defined as in \eqref{triang}. By \eqref{bounduv}, we  have
\begin{multline*}
	|\nabla w_\vep(x)|^2\le \frac{C}{\vep^2}(|\ud u_\vep(i,i +\vep e_1) |^2+|\ud u_\vep(i+\ep e_1,i + \vep e_1+\vep e_2) |^2)\\
	\leq  \frac{C\,C_\n}{\vep^2}(|\ud v_\vep(i,i +\vep e_1) |^2+|\ud v_\vep(i+\ep e_1,i + \vep e_1+\vep e_2) |^2),
\end{multline*}
which combined with the fact that $w_\vep(i)=v_\vep(i)$ and the Mean Value Theorem, yields
\begin{multline}\label{l2bound}
	|w_\vep(x)-\hat v_\vep(x)|^2\le 2(|w_\vep(x)-w_\vep(i)|^2+|\hat v_\vep(x)-v_\vep(i)|^2)\\
	\leq C (|\ud v_\vep(i,i +\vep e_1) |^2+|\ud v_\vep(i+\ep e_1,i + \vep e_1+\vep e_2)|^2).
\end{multline}
Using the same argument one can show that for any $x\in T_i^+$
\begin{equation*}
	|w_\vep(x)-\hat v_\vep(x)|^2
	\leq C (|\ud v_\vep(i+\ep e_2,i +\vep e_1+\ep e_2) |^2+|\ud v_\vep(i,i +\vep e_2) |^2).
\end{equation*}

%

By integrating \eqref{l2bound} and using \eqref{massJvep}, we obtain
\begin{multline*}
\int_{\Omega}|w_\vep(x)-\hat v_\vep(x)|^2\ud x\le \int_{\Omega_\ep\setminus J_\ep}|w_\vep(x)-\hat v_\vep(x)|^2\ud x+\int_{J_\ep}|w_\vep(x)-\tilde v_\vep(x)|^2\ud x\\
\le C\ep^2 XY_{\ep}(v_\ep,\Omega_\ep\setminus J_\ep)+C\ep|\log\ep|,
\end{multline*}
 which, sending $\vep\to 0$, yields $\hat u^\n_\ep=w_\ep\to v$ in $L^2(\Omega;\R^2)$.
Now, since $\hat u_\ep\to u$ in $L^1(\Omega;\R^2)$, we clearly have that $\hat u^\n_\ep\to u^{\n}$ in $L^1(\Omega;\R^2)$,  so that we conclude
$u^\n=v$. 
\qed
\subsection{Proof of (ii): $\Gamma$-liminf inequality}

We can assume without loss of generality  that \eqref{energyBound} holds and that $\F^{\nn}_\vep(\vph_\vep) - M\pi|\log\vep|$ converges; then, by the Compactness result $(i)$ of Theorem \ref{mainthm}, 
$\hat{u}_\vep \weakly u$  in $SBV_\loc^2(\Omega\setminus (\supp Ju^\n);\R^2)$ for some $u \in \DD$.
Set $\mu:=Ju^\n$. Since $u\in\DD$, $\mu= \pi \sum_{i=1}^M d_i\delta_{x_i}$ with $|d_i|=1$ and $x_i\in\Omega$ for any $i=1,\ldots,M$.
Let $\rho>0$ be such that the balls $B_\rho(x_i)$ are pairwise disjoint and contained in $\Omega$. 
As $S_u$ is rectifiable, it is contained,  up to a $\mathcal H^1$-negligible set, in   a countably  union $\bigcup_{i=1}^{\infty} C_i$ of compact $C^1$-curves contained in $\Om$. For every $N \in \N$, we define $S_u^N := S_u\cap \bigcup_{i=1}^N C_i$. Then, for any given $\delta > 0$ we define   the $\delta$-tube around $S_u^N$ in $\Omega_\rho(\mu)$ (see Figure \ref{tubefig}) as 
\[
  \tube := \{x \in \Omega_\rho(\mu) \colon \dist(x, S_u^N) < \delta\}.
\]
Moreover, we set $\tubep:= (\tube)_\ep$. 
Notice that for  all $\ep>0$ and for $\delta$ small enough,   $\Om\setminus \tubep $ has Lipschitz continuous boundary.
\begin{figure}[h!]
	\centering
	\def\svgwidth{0.4\textwidth}
	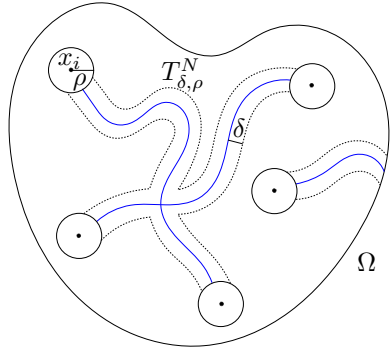
	\caption{ The blue lines represent  $C^1$-curves $C_1,\ldots,C_N$ and the dashed lines represent the boundary of the tube $T^N_{\delta,\rho}$.}
	\label{tubefig}
\end{figure}

By \eqref{comparxy}, Theorem \ref{gammal2} (c) and using the very definition of $\nren$ in \eqref{renormalizedEnergy} and \eqref{Wlocal}, we have that
\begin{multline}\label{vorticityEnergy}
		\liminf_{\vep \to 0} \F^{\nn}_\vep(\vph_\vep,\Omega) - M\pi|\log\vep|\\ 
		\geq \liminf_{\vep \to 0} (\F^\eff_\vep(\vth_\vep,\Omega \setminus \tube) - M\pi|\log\vep|) + \liminf_{\vep \to 0} \F^{\nn}_\vep(\vph_\vep,\tube)\\
\geq \mathcal{W}(v,\Omega \setminus \tube) + M \gamma+ \liminf_{\vep \to 0} \F^{\nn}_\vep(\vph_\vep, \tubep)\\
		= \mathcal{W}(v) + M \gamma - \frac{1}{2} \int_{\tube} |\nabla v|^2 \ud x + \liminf_{\vep \to 0} \F^{\nn}_\vep(\vph_\vep, \tubep).
\end{multline}

We claim that
\begin{equation}\label{lengthEnergy}
	\int_{S_u^N \cap \Omega_\rho(\mu)}|\nu_u|_1 \ud\mathcal{H}^1 \leq \liminf_{\vep \to 0} \F^{\nn}_\vep(\vph_\vep,\tubep).
\end{equation}
Using the  claim, the $\Gamma$-liminf inequality follows by sending $\ep\to 0$,  $\delta\to 0$, $\rho\to 0$  and $N\to +\infty$ in \eqref{vorticityEnergy} and in \eqref{lengthEnergy}.

It remains to prove the claim \eqref{lengthEnergy}. Fix $0<t<1$ and let $\Theta_t:\R\times \R\to \R$ be defined by
$$
\Theta_t(a,b):=\left\{\begin{array}{ll}
(t + \frac{1-t}{2\sin\frac\pi {\n}}|a-b|) &\textrm{if }|a-b|\le 2\sin\frac{\pi}{\n}\\
1&\textrm{if }|a-b|\ge 2\sin\frac{\pi}{\n}.
\end{array}\right.
$$
It is straightforward to check that $\Theta_t$ is positive, symmetric and satisfies the triangular inequality  \eqref{triangineq}. Let $\{A_n\}_{n\in\N}$ be a sequence of open sets with $A_n\subset \overline {A_{n}}\subset A_{n+1}$ for all $n\in\N$ and such that $\bigcup_{n\in\N} A_n = \tube$. 
Since $|u^+-u^-|\ge 2\sin\frac{\pi}{\n}$ on $S_u$, we have that $\Theta_t(u^+,u^-)=1$ on $S_u$, whence, recalling also   Theorem \ref{lscSBV},  we get
\begin{multline}\label{BVEllipticityAppl}
 \int_{S_u^N \cap \Omega_\rho(\mu)} |\nu_u|_1 \ud\mathcal{H}^1 
 \leq \int_{S_u \cap \tube}|\nu_u|_1 \ud\mathcal{H}^1= \int_{S_u \cap \tube}\Theta_t(u^+,u^-)|\nu_u|_1 \ud\mathcal{H}^1\\
 = \lim_{n\to +\infty} \int_{S_u \cap A_n}\Theta_t(u^+,u^-)|\nu_u|_1 \ud\mathcal{H}^1
\leq \lim_{n\to +\infty} \liminf_{\vep \to 0} \int_{S_{\hat u_\vep} \cap A_n}\Theta_t( \hat u_\vep^+, \hat u_\vep^-)|\nu_{ \hat u_\vep}|_1 \ud\mathcal{H}^1\\ 
\leq \liminf_{\vep \to 0} \int_{S_{\hat u_\vep} \cap T^N_{\delta - \sqrt{2} \ep, \rho, \ep}}\Theta_t( \hat u_\vep^+, \hat u_\vep^-)|\nu_{ \hat u_\vep}|_1 \ud\mathcal{H}^1.
\end{multline}

Therefore, in order to get \eqref{lengthEnergy} it is enough to prove that 
\begin{equation}\label{boh}
	\liminf_{\ep\to 0}\F^{\nn}_\vep(\vph_\vep,\tubep) + Ct \geq 
	\liminf_{\ep\to 0}\int_{S_{\hat u_\vep} \cap       T^N_{\delta - \sqrt{2} \ep, \rho, \ep}  }\Theta_t(\hat u_\vep^+,\hat u_\vep^-)|\nu_{\hat u_\vep}|_1 \ud\mathcal{H}^1,
\end{equation}
for some $C>0$.
To this purpose, we set
\begin{eqnarray*}
&T_{\ep}^\negl := \bigcup \left\{i+\vep {Q} \colon i\in \Om_\vep^2,\, \di(\vph_\vep(k) - \vph_\vep(j),\textstyle \frac{2\pi} \n \Z) > \sqrt[3]\vep\,\textrm{ for some }\right.\\
		&\left.(j,k) \in  \Om_\ep^1 \text{ with } 
		\{j,k\} \cap \partial (i+\vep {Q}) \neq \emptyset \right\}.
\end{eqnarray*}

We first show that
\begin{equation}\label{neglig}
\liminf_{\ep\to 0}\int_{S_{\hat u_\vep}\cap T_{\ep}^\negl \cap \Om_\rho(\mu)}\Theta_t(\hat u_\vep^+,\hat u_\vep^-)|\nu_{\hat u_\vep}|_1 \ud\mathcal{H}^1=0.
\end{equation}

Indeed, by  the energy bounds \eqref{comparxy} and \eqref{liminfloc}, there exists $C_\rho>0$ such that
$$
XY_\ep(v_\ep, T_{\ep}^\negl ) \leq \F^\eff_\vep(\vth_\vep,\Omega_\rho(\mu))\leq C_\rho.
$$
Therefore, by the definition of $T_{\ep}^\negl$  there exists $C>0$ such that 
\begin{multline*}
C	\,  \sharp(T_{\ep}^\negl )_{\vep}^1 (1-\cos(\n\sqrt[3]\vep))
\\
\leq\sharp \{ (j,k)\in (T_{\ep}^\negl )_{\vep}^1: \, \di(\vph_\vep(k) - \vph_\vep(j),\textstyle \frac{2\pi} \n \Z) > \sqrt[3]\vep\} 
(1-\cos(\n\sqrt[3]\vep)) 
\\
\leq XY_\ep(v_\ep, T_{\ep}^\negl ) \leq C_\rho,
\end{multline*}
which by Taylor expansion yields 
 $\vep \sharp (T_{\ep}^\negl )_{\vep}^1 \to 0$ and eventually 
\begin{equation}\label{negligiblejump}
\lim_{\ep\to 0}\mathcal{H}^1({S_{ \hat u_\vep} \cap  T_{\ep}^\negl  } \cap \Om_\rho(\mu)) = 0.
\end{equation} 
Then \eqref{neglig} follows by noticing that $\Theta_t$ is uniformly bounded by $1$.

Consider the map $\mathcal J: (\R^2)_\ep^1 \to  (\R^2)_\ep^1$ defined by
$$
\mathcal J ((i,i+\ep e_1))=(i-\ep e_2,i),\quad \mathcal J((i,i+\ep e_2))=(i-\ep e_1,i) \qquad \text{ for all } i\in \ep\Z^2,
$$
and $\mathcal J(i,j) = - \mathcal J(j,i)$ for all $(i,j) \in (\R^2)_\ep^1$. 
Moreover, for all $(i,j)\in  (\R^2)_\ep^1$, let $(\imath, \jmath):= \mathcal J(i,j)$. Now, 
set  
$$
\tilde T^N_{\delta,\rho,\ep}:=\overline{T^N_{\delta,\rho,\ep} \setminus T_{\ep}^\negl}, 
\qquad
\tilde T^N_{\delta - \sqrt{2} \ep,\rho,\ep}:=\overline{T^N_{\delta - \sqrt{2} \ep,\rho,\ep} \setminus T_{\ep}^\negl}.
$$
  We show that for any  $(i,j) \in  (\tilde{T}^N_{\delta,\rho,\ep})_{\ep}^1$ and for any $x\in [i,j]$, there holds
\begin{equation}\label{maxest1}
|u_\ep(\jmath)-u_\ep(\imath)|-\sqrt[3]{\ep}\le| \hat u_\ep^+(x)-\hat u_\ep^-(x)|\le |u_\ep(\jmath)-u_\ep(\imath)|+\sqrt[3]{\ep}.
\end{equation}

\begin{figure}[h!]
	\centering
	\def\svgwidth{0.7\textwidth}
	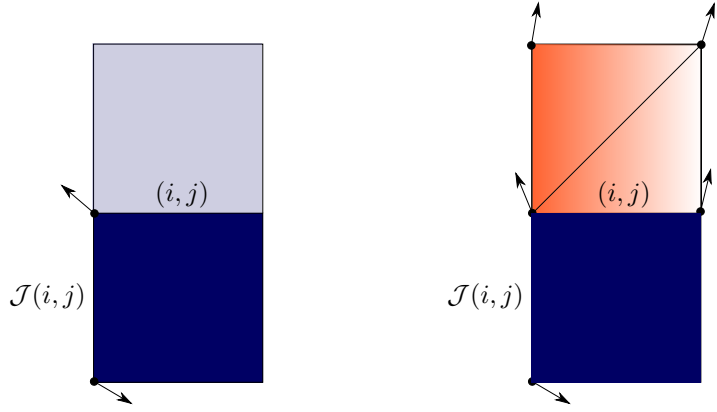
	\caption{On the left hand side: two adjacent jump cells. On the right hand side: a lower jump cell adjacent to a non-jump cell.}
\label{edge_vs_pair}
\end{figure}

We prove  \eqref{maxest1} only in the case $j=i+ \ep e_1$,  the proof in the case $j=i+\ep e_2$   being analogous.  
Notice that it is trivial whenever $i+\ep Q$ and $i-\ep e_2+\ep Q$ belong to the family of   jump cells $JC_{\ffi_\ep}$ (see Figure \ref{edge_vs_pair}),  since in this case for any $x\in [i,i+\ep e_1]$
$$
|\hat u_\ep^+(x)-\hat u_\ep^-(x)|=|u_\ep(i)-u_\ep(i-\ep e_2)| =|u_\ep(\jmath)-u_\ep(\imath)|.
$$
Analogously, \eqref{maxest1} is trivial whenever both $i+\ep Q$ and $i-\ep e_2+\ep Q$ are not jump cells, since in this case 
$|\hat u_\ep^+(x)-\hat u_\ep^-(x)|=0$ and  (recall $(i,j)\notin T_\ep^\negl$) $|u_\ep(\jmath)-u_\ep(\imath)| \le \sqrt[3]{\ep}$.
Now, without loss of generality we assume that (see Figure \ref{edge_vs_pair}) $i+\ep Q$ is not a jump cell while $i-\ep e_2+\ep Q$ is a jump cell (the proof being fully analogous in the case that $i-\ep e_2+\ep Q$ is not a jump cell while $i+\ep Q$ is a jump cell). Furthermore, we prove only the second inequality in \eqref{maxest1} since the first one follows similarly. 
Let $\bar s\in [0,1]$ be such that $x= i+\bar s \ep e_1$. 
Then, using that $i+\ep Q\subset \tilde T^N_{\delta,\rho,\ep}$ and $i+\ep Q\notin JC_{\ffi_\ep}$, 
we have
\begin{multline}\label{nobothjump}
 |\hat u_\ep^+(i+\bar s\ep e_1)-\hat u_\ep^-(i+\bar s\ep e_1)|\\
= |u_\ep(i)+\bar s(u_\ep(i+\ep e_1)-u_\ep(i))-u_\ep(i-\ep e_2)|\\
\le |u_\ep(i)-u_\ep(i-\ep e_2)|+|u_\ep(i+\ep e_1)-u_\ep(i)|\\
\le |u_\ep(i)-u_\ep(i-\ep e_2)|+ 2\sin\frac{\sqrt[3]\vep}{2} 
\le |u_\ep(i)-u_\ep(i-\ep e_2)|+\sqrt[3]\ep.
\end{multline}
In conclusion we have proven that \eqref{maxest1} holds true.

Notice now that  if $(i,j) \in  (\tilde{T}^N_{\delta,\rho,\ep})_{\ep}^1$,   
then $(i,j) \notin (T_{\ep}^\negl)_\ep^1$,  and hence $(\imath,\jmath)=\mathcal J(i,j)$ satisfies 
 either
\begin{equation}\label{closeto0}
\di(\ffi_\ep(\jmath)-\ffi_\ep(\imath),2\pi\Z)\le \sqrt[3]{\ep}
\end{equation}
or
\begin{equation}\label{closetojump}
\di(\ffi_\ep(\jmath)-\ffi_\ep(\imath),\frac{2l'\pi}{\n})\le \sqrt[3]{\ep},
\end{equation}
for some $l'\in \Z\setminus \n\Z$.
If $(\imath,\jmath)$  satisfies \eqref{closeto0}, then, by \eqref{maxest1}, 
$$
\max_{x\in[i,j]}|\hat u^+_\ep(x)-\hat u^-_\ep(x)|\le |u_\ep(\jmath)-u_\ep(\imath)|+\sqrt[3]{\ep}\\
=2|\sin\textstyle \frac{\ffi_\ep(\jmath)-\ffi_\ep(\imath)}{2}|+\sqrt[3]{\ep}\le 2\sqrt[3]{\ep}.
$$
On the other hand, if $(\imath,\jmath)$ satisfies \eqref{closetojump} for some $l'\in\Z\setminus\n\Z$, then, again by \eqref{maxest1},  and using Taylor expansion, for any $x\in[i,j]$ and for $\ep$ sufficiently small, we have
\begin{multline*}
2|\textstyle \sin\frac{l'\pi}{\n}|-2\sqrt[3]{\ep}\le 2|\sin\frac{\ffi_\ep(\jmath)-\ffi_\ep(\imath)}{2}|-\sqrt[3]{\ep}\\
=|u_\ep(\jmath)-u_\ep(\imath)|-\sqrt[3]{\ep}\le  |\hat u^+_\ep(x)- \hat u_\ep^-(x)|\le |u_\ep(\jmath)-u_\ep(\imath)|+\sqrt[3]{\ep}\\
\le2|\textstyle \sin\frac{l'\pi}{\n}|+2\sqrt[3]{\ep}.
\end{multline*}

It follows that, if $\ep$ is sufficiently small, the jump bonds of $\hat u_\ep$  in $\tilde T^N_{\delta,\rho,\ep}$, and, a fortiori,  the jump bonds of $\hat u_\ep$  in  $\tilde T^N_{\delta - \sqrt{2}\ep,\rho,\ep}$ lie in one of the following sets
\begin{eqnarray*}
&&I_{\delta,\rho,\ep}^{s}:=\{(i,j)\in(\tilde T^N_{\delta - \sqrt{2}\ep,\rho,\ep})_\ep^1\,:\,\max_{x\in[i,j]}|\hat u^+_\ep(x)-\hat u^-_\ep(x)|\le 2\sqrt[3]{\ep}\},\\
&&I_{\delta,\rho,\ep}^{b}:=\{(i,j)\in(\tilde T^N_{\delta - \sqrt{2}\ep,\rho,\ep})_\ep^1\,:\,\left|\textstyle |\hat u^+_\ep(x)-\hat u^-_\ep(x)|-2\sin\frac{l\pi}{\n}\right|\le 2\sqrt[3]{\ep}\\
&&\phantom{I_{\delta,\rho,\ep}^{bj}:=}\textrm{ for any }x\in[i,j], \textrm{ for some }l=1,\ldots,  {\n} -1  \}.
\end{eqnarray*}
Set 
$$
E^s:= \bigcup_{(i,j)\in I_{\delta,\rho,\ep}^{s}}[i,j], \qquad E^b:= \bigcup_{(i,j)\in I_{\delta,\rho,\ep}^{b}}[i,j].
$$
By the very definition of $\Theta_t$ and by the uniform bound of $\mathcal H^1(S_{\hat u_\ep})$ in \eqref{boundjump},
for $\ep$ small enough we have
\begin{equation}\label{thirdint}
\int_{S_{\hat u_\ep}\cap E^s}\Theta_t(\hat u_\ep^+,\hat u_\ep^-)|\nu_{\hat u_\ep}|_1\ud \mathcal H^1\le (\textstyle t+\frac{1-t}{\sin\frac{\pi}{\n}}\sqrt[3]{\ep})\mathcal H^1(S_{\hat u_\ep})\le C( t + \sqrt[3]{\ep}).
\end{equation}
In order to conclude the proof of \eqref{boh}, we first show that
\begin{equation}\label{firstint}
\liminf_{\ep\to 0}\int_{S_{\hat u_\ep}\cap E^b}\Theta_t(\hat u_\ep^+,\hat u_\ep^-)|\nu_{\hat u_\ep}|_1\ud \mathcal H^1\le \liminf_{\ep\to 0}\F_\ep^\n(\ffi_\ep, T^N_{\delta,\rho,\ep}) .
\end{equation}
To this purpose, we set
\begin{multline*}
\mathcal{I}^b_{\delta,\rho,\ep}:=\{(j,k)\in (T^N_{\delta,\rho,\ep})_\ep^1\setminus (T^\negl_{\ep})_\ep^1\,:\,\di(\textstyle \ffi_\ep(k)-\ffi_\ep(j),\frac{2l'\pi}{\n})\le\sqrt[3]{\ep}\\
\textrm{ for some }l'\in\Z\setminus\n\Z\},
\end{multline*}
Notice that $ \mathcal J (I^b_{\delta,\rho,\ep})\subset  (T^N_{\delta,\rho,\ep})_\ep^1$.
Now we show that, for $\ep$ small enough,
$ \mathcal J (I^b_{\delta,\rho,\ep})\subset \mathcal I^b_{\delta,\rho,\ep}$.
Let $(i,j)\in I^b_{\delta,\rho,\ep}$ and let $l\in 1,\ldots, \n -1$ be such that, for all $x\in [i,j]$ we have 
\begin{equation}\label{jumpuhat}
|\hat u^+_\ep(x)-\hat u^-_\ep(x)|\in[\textstyle 2\sin\frac{l\pi}{\n}-2\sqrt[3]\ep,2\sin\frac{l\pi}{\n}+2\sqrt[3]\ep].
\end{equation}
On one hand, by \eqref{maxest1} and \eqref{jumpuhat}, we have that
\begin{equation}\label{uno}
2\sin\frac{l\pi}{\n}-3\sqrt[3]\ep\le |u_\ep(\jmath)-u_\ep(\imath)|\le 2\sin\frac{l\pi}{\n} + 3\sqrt[3]\ep.
\end{equation}
On the other hand, since $(i,j)\notin (T^\negl_{\ep})_\ep^1$, there exists $l'\in\Z$ such that 
\begin{equation}\label{due}
\di(\textstyle \ffi_\ep(\jmath)-\ffi_\ep(\imath),\frac{2l'\pi}{\n})\le\sqrt[3]\ep.
\end{equation}
Therefore, by \eqref{uno} and \eqref{due} and Taylor expansion, it immediately follows that, for $\ep$ small enough, $l'\in l+\Z$ and hence $(\imath,\jmath)\in \mathcal I^b_{\delta,\rho,\ep}$. We have concluded the proof that $ \mathcal J (I^b_{\delta,\rho,\ep})\subset \mathcal I^b_{\delta,\rho,\ep}$.

Since the map $\mathcal J$ is  injective,  we have
\begin{equation}\label{jpvsjb}
\sharp I^{b}_{\vep,\delta,\rho}=\sharp \mathcal J( I^{b}_{\delta,\rho,\ep}) \leq \sharp \mathcal{I}^b_{\delta,\rho,\ep}.
\end{equation}
Now, using that $\Theta_t\le 1$, we obtain
\begin{equation*}
	\int_{S_{\hat u_\ep}\cap E^b}\Theta_t(\hat u_\vep^+,\hat u_\vep^-)|\nu_{\hat u_\vep}|_1\ud \mathcal H^1
\leq \vep \sharp I_{\delta,\rho,\ep}^b\le \ep \sharp \mathcal I^b_{\delta,\rho,\ep}\le \F_\vep^{\nn}(\vph_\vep,\tubep),
\end{equation*}
which is exactly \eqref{firstint}.

By \eqref{neglig}, \eqref{thirdint} and \eqref{firstint}, we deduce 
\begin{multline*}
\liminf_{\ep\to 0} \int_{S_{\hat u_\vep} \cap T^N_{\delta - \sqrt{2} \ep, \rho, \ep} }
\Theta_t(\hat u_\vep^+,\hat u_\vep^-)|\nu_{\hat u_\vep}|_1 \ud\mathcal{H}^1
\\
=
\liminf_{\ep\to 0} \int_{(S_{\hat u_\vep} \cap \tilde T^N_{\delta - \sqrt{2} \ep, \rho, \ep})  }
\Theta_t(\hat u_\vep^+,\hat u_\vep^-)|\nu_{\hat u_\vep}|_1 \ud\mathcal{H}^1
\\
\le
\liminf_{\ep\to 0} \int_{S_{\hat u_\vep} \cap E^b }
\Theta_t(\hat u_\vep^+,\hat u_\vep^-)|\nu_{\hat u_\vep}|_1 \ud\mathcal{H}^1 + C( t + \sqrt[3]{\ep})
\\
\le\liminf_{\ep\to 0} \F_\vep^{\nn}(\vph_\vep,\tubep) + Ct,
\end{multline*}
which concludes the proof of 
the inequality in \eqref{boh} and of the $\Gamma$-liminf inequality.

\subsection{Proof of (iii): $\Gamma$-limsup inequality}

We start this subsection by proving two lemmas that will be useful in the proof of the $\Gamma$-limsup inequality  \eqref{gammalimsup}.
The first lemma concerns with the first-order core energy induced by a singularity with degree $d=\pm 1$. 
We recall that $\theta(\cdot)$ denotes the angular polar coordinate in $\R^2\setminus\{0\}$.

\begin{lemma}\label{corelemma}
Let $d\in\{-1,+1\}$. For any $\ep, \sigma>0$, set
\begin{equation}\label{truegamma}
	\gamma'(\vep,\sigma):= \min_{\vph \in \mathcal{AF}_\vep(B_\sigma)} \left\{\F^{\nn}_\vep(\vph,B_\sigma): e^{i\n\vph(\cdot)} = e^{i d\theta(\cdot) }\text{ on } \partial_\vep B_\sigma\right\};
\end{equation}
there holds 
\begin{equation}\label{samelimit}
\gamma=\limsup_{\sigma\to 0}\limsup_{\ep\to 0}\gamma'(\vep,\sigma)-\pi|{\textstyle \log\frac\ep\sigma}|=\liminf_{\sigma\to 0}\liminf_{\ep\to 0}\gamma'(\vep,\sigma)-\pi|\textstyle \log\frac\ep\sigma|,
\end{equation}
where $\gamma$ is defined in \eqref{gamma}.
\end{lemma}

\begin{proof}
We prove \eqref{samelimit} only in the case $d=1$, the proof in the case $d=-1$  being fully analogous.

By \eqref{comparxy}, $\gamma'(\ep,\sigma)\ge \gamma(\ep,\sigma)$ for any $\ep,\,\sigma>0$, whence
$$
\liminf_{\sigma\to 0}\liminf_{\ep\to 0}\gamma'(\vep,\sigma)-\pi|\textstyle \log\frac\ep\sigma|\ge\gamma.
$$
In order to prove \eqref{samelimit}, it is enough to show that
\begin{equation}\label{remaintoprove}
\gamma'(\ep,\sigma)\le\gamma(\ep,\sigma)+r(\ep,\sigma),
\end{equation}
where $\limsup_{\sigma\to 0}\limsup_{\ep\to 0}r(\ep,\sigma)=0$.

To this purpose fix $\sigma>0$. 
For any $\ep>0$, let $\vth_{\ep,\sigma}$ be a solution of the minimization problem in \eqref{gammaep} and set  $\mu_{\ep,\sigma}:=\mu(\vth_{\ep,\sigma})$. By \eqref{dacitare}, (for $\ep<\sqrt 2 \sigma$) we have 
$\mu_{\ep,\sigma}(B_\sigma)=1$.
Let $\vth'_{\ep,\sigma}:(B_{2\sigma})_\ep^0\to\R$ be the extension of $\vth_{\ep,\sigma}$ defined by
$$
\vth'_{\ep,\sigma}(i):=\left\{\begin{array}{ll}
\vth_{\ep,\sigma}(i)&\textrm{if }i\in (B_{\sigma})_\ep^0\\
\theta(i)&\textrm{otherwise},
\end{array}\right.
$$
where $\theta$ is the angular polar coordinate. Notice that $\mu(\vth'_{\ep,\sigma})=\mu_{\ep,\sigma}$ so that
\begin{equation}\label{richiama}
\mu(\vth'_{\ep,\sigma})(B_{2\sigma}\setminus B_\sigma)=0\qquad\textrm{and}\qquad \mu(\vth'_{\ep,\sigma})(\bar B_{\sigma})=1.
\end{equation}
We  show that there exists $\xi\in \bar B_\sigma$ such that
\begin{equation}\label{convtoxi}
\|\mu_{\ep,\sigma}-\delta_\xi\|_{\fla(B_{2\sigma})}\to 0.
\end{equation}
Notice that, by standard interpolation estimates \cite{C},
$$
\lim_{\ep\to 0} \F_\ep^\eff(\vth'_{\ep,\sigma},B_{2\sigma}\setminus (B_{\sigma})_\ep)=\pi\log 2;
$$
therefore, for sufficiently small $\ep$, we have
\begin{multline*}
\textstyle \F_\ep^\eff(\vth'_{\ep,\sigma},B_{2\sigma})-\pi|\log\frac\ep\sigma|\\
\le  \F_\ep^\eff(\vth'_{\ep,\sigma},B_{2\sigma}\setminus (B_\sigma)_\ep)+ \F_\ep^\eff(\vth'_{\ep,\sigma},B_{\sigma})-\pi|\log\frac\ep\sigma| 
\le 2\pi\log 2+ 2\gamma.
\end{multline*}
 By Theorem \ref{gammal2}, we have that, up to a subsequence,
\begin{equation}\label{flatconv}
\|\mu(\vth'_{\ep,\sigma})-\mu\|_{\fla(B_{2\sigma})}\overset{\ep\to 0}{\longrightarrow}0,
\end{equation}
where either $\mu\equiv 0$ or $\mu=\pi d\delta_\xi$ for some  $d\in\{-1,1\}$ and $\xi\in B_{2\sigma}$.

%
Since, in view of \eqref{richiama}, we have $\mu(B_{2\sigma}\setminus B_{\sigma})=0$ and $\mu(\bar B_\sigma)=\pi$, we get that $\mu=\pi \delta_\xi$ with $\xi\in\bar B_\sigma$,  which concludes the proof of \eqref{convtoxi}.

By \eqref{minimalconn} there exist integer 1-currents $T_{\ep,\sigma}$, whose support is a finite union of segments,  with $\partial  T_{\ep,\sigma}\res B_{2\sigma}=\mu_{\ep,\sigma}-\delta_\xi$ and
\begin{equation}\label{Tepsig}
\lim_{\ep\to 0}|T_{\ep,\sigma}|= 0.
\end{equation}
Let moreover $T^\xi$ be an integer $1$-current with $\partial T^\xi\res B_{2\sigma}=\delta_\xi$ 
and $|T^\xi|\le 4\sigma$ 
(such a current can be constructed, for instance, by considering an oriented segment  $I^\xi$ joining $\xi$ with a point in $\partial B_ {2\sigma}$).

Set $\hat T_{\ep,\sigma}:=T_{\ep,\sigma}+T^\xi$; by construction $\partial \hat T_{\ep,\sigma}\res B_{2\sigma}=\mu_{\ep,\sigma}$ and, by \eqref{Tepsig}, 
\begin{equation}\label{lengthThat}
\limsup_{\ep\to 0}|\hat T_{\ep,\sigma}|\le \limsup_{\ep\to 0}|T_{\ep,\sigma}|+\limsup_{\ep\to 0}| T^\xi| \le 4\sigma.
\end{equation}
Set
$$
V:=\bigcup_{\newatop{i\in (B_\sigma)_\ep^2}{(i+\ep Q)\cap\,\supp \hat T_{\ep,\sigma}\neq\emptyset}}(i+\ep Q)
$$
and $E:=B_{\sigma}\setminus V$. 
Let $E_{1},\ldots,E_{K_{\ep,\sigma}}$ denote the  connected components of $E_\ep$.
Let $A\subset B_\sigma$ be such that $A_\ep$ is simply connected and $\partial A_\ep\cap\supp \hat T_{\ep,\sigma}=\emptyset$. Then, it is easy to see that $\mu_{\ep,\sigma}(A_\ep)=0$.
By \eqref{dacitare}, it follows that
\begin{equation}\label{disctopnull}
\sum_{l=1}^{L-1}\ud^\el \vth_{\ep,\sigma}(i^l,i^{l+1})+\ud^\el\vth_{\ep,\sigma}(i^L,i^1)=0,
\end{equation}
where $L:=\sharp \partial_\ep A$ and $\partial_\ep A:=\{i^1,\ldots,i^L\}$, with $(i^l,i^{l+1})\in A_\ep^1$ for any $l=1,\ldots,L-1$. Therefore, it is well-defined the function $\bar \theta_{\ep,\sigma}:E_\ep^0\to\R$ constructed as follows:
For any $k=1,\ldots,K_{\ep,\sigma}$ fix $i^0_k\in (E_{k})_{\ep}^0$ and for any $i\in (E_{k})_{\ep}^0$, set
\begin{equation}\label{tetaprimo}
\bar\vth_{\ep,\sigma}(i): = \vth_{\ep,\sigma}(i^0_k)+\sum_{l=1}^L\ud^\el\vth_{\ep,\sigma}(i_k^{l-1},i_k^{l}),
\end{equation}
where  $\{i^0_k,\,i^1_k,\,\ldots,i_k^L(=i)\}\subset (E_k)_\ep^0$ is  such that $(i_k^{l-1},i_k^l)\in (E_{k})_\ep^1$ for $l=1,\ldots,L$.

\begin{figure}[h!]
		\centering
	\includegraphics[width=0.4\textwidth]{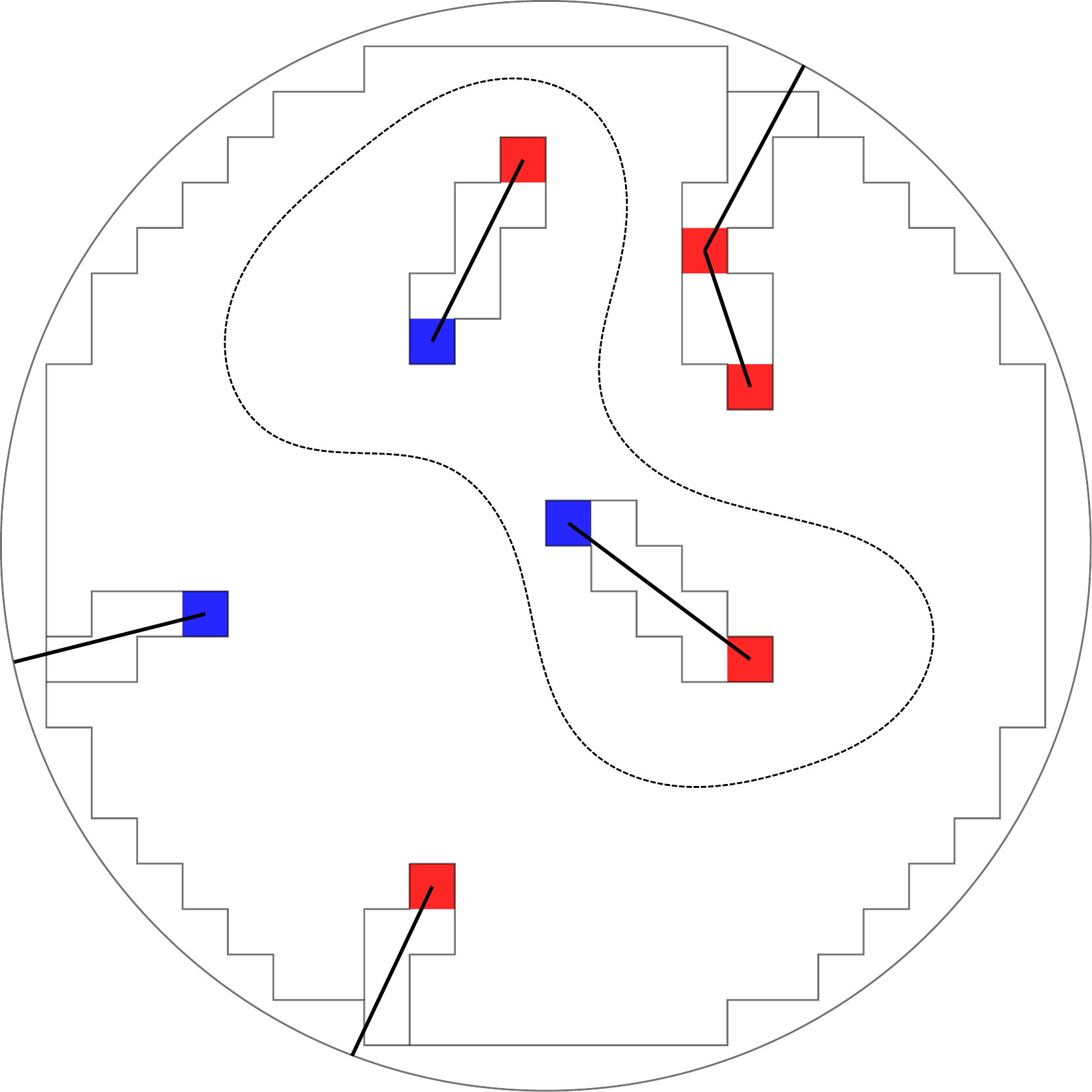}
\caption{Closed path in the core $B_{\sigma}$: The line segments denote the support of $T_{\ep,\sigma}$. All cells having positive (resp., negative) vorticity $\alpha_{\vth_{\ep,\sigma}}$ are colored red  (resp.,  blue). It can be easily seen that any closed path in $E$ is surrounding the same amount red and blue cells, implying that the definition of $\bar\vth_{\ep,\sigma}$ in \eqref{tetaprimo} is well-posed.}
\label{corefig}
\end{figure}
Indeed, by \eqref{disctopnull}, one can easily check that $\bar\vth_{\ep,\sigma}(i)$ does not depend on the path $i^0_k,\,i^1_k,\,\ldots,i_k^L(=i)$ (see also Figure \ref{corefig}). Moreover, since  for any $(i,j)\in E_\ep^1$ there exists $k=1,\ldots,K_{\ep,\sigma}$ such that $(i,j)\in (E_k)_\ep^1$, there  holds
\begin{equation}\label{onlyela}
\ud\bar\vth_{\ep,\sigma}(i,j)=\ud^\el\vth_{\ep,\sigma}(i,j)\qquad\textrm{ for all }(i,j)\in E_\ep^1.
\end{equation}
Furthermore, by the very definition of $P$, it immediately follows that 
\begin{equation}\label{equaltheta}
\bar\vth_{\ep,\sigma}(i)-\vth_{\ep,\sigma}(i)\in 2\pi\Z\qquad\textrm{for any }i\in E_\ep^0.
\end{equation}
For any $i\in (B_\sigma)_\ep^0$, we set
\begin{equation}
\bar\ffi_{\ep,\sigma}(i):=\left\{\begin{array}{ll}
\frac{\bar\vth_{\ep,\sigma}(i)}{\n}&\textrm{if }i\in E_{\ep}^0\\
\frac{\vth_{\ep,\sigma}(i)}{\n}&\textrm{otherwise}.
\end{array}\right.
\end{equation}
By \eqref{equaltheta}, $\bar\ffi_{\ep,\sigma}$ is a competitor for the minimum problem \eqref{truegamma}.

In order to prove \eqref{remaintoprove}, it is enough to show that
\begin{equation}\label{remain2}
\F_\ep^{\nn}(\bar\ffi_{\ep,\sigma},B_\sigma)\le \F^\eff_{\ep}(\vth_{\ep,\sigma})+r(\ep,\sigma),
\end{equation}
with $\limsup_{\sigma\to 0}\limsup_{\ep\to 0}r(\ep,\sigma)=0$.

In order to get \eqref{remain2} we first notice that by \eqref{onlyela}, for any $(i,j)\in E_\ep^1$, there holds
$$
-\frac\pi \n\le\bar\ffi_{\ep,\sigma}(j)-\bar\ffi_{\ep,\sigma}(i)\le\frac\pi \n,
$$
whence, using \eqref{equaltheta} and the very definition of $f_\ep^{\nn}$ in \eqref{potential}, we obtain
\begin{equation}\label{elasticcore}
\frac 1 2\sum_{{(i,j)\in E_\ep^1\setminus V_\ep^1}}f_{\ep}^{\nn}(\bar\ffi_{\ep,\sigma}(j)-\bar\ffi_{\ep,\sigma}(i))=\frac 1 2\sum_{{(i,j)\in E_\ep^1\setminus V_\ep^1}}f(\vth_{\ep,\sigma}(j)-\vth_{\ep,\sigma}(i)).
\end{equation}
Moreover,  again  by the definition of $f_\ep^{\nn}$, we have
\begin{multline}\label{jumpscore}
\frac 1 2\sum_{(i,j)\in V_\ep^1}f_{\ep}^{\nn}(\bar\ffi_{\ep,\sigma}(j)-\bar\ffi_{\ep,\sigma}(i))
\le \frac 1 2\sum_{(i,j)\in V_\ep^1}f(\vth_{\ep,\sigma}(j)-\vth_{\ep,\sigma}(i))+2\ep\sharp V_\ep^0.
\end{multline}

Set  $r(\ep,\sigma):=2\ep\sharp V_\ep^0$, by construction and by \eqref{lengthThat}, we have that 
\begin{equation}\label{remainder}
\limsup_{\sigma\to 0}\limsup_{\ep\to 0}r(\ep,\sigma)\le C \limsup_{\sigma\to 0}\limsup_{\ep\to 0}|\hat T_{\ep,\sigma}|\le C\limsup_{\sigma\to 0}|T^\xi|=0.
\end{equation}
In the end, by \eqref{elasticcore}, \eqref{jumpscore} and \eqref{remainder}, we get
\begin{eqnarray*}
&&\F_\ep^{\nn}(\bar\ffi_{\ep,\sigma},B_\sigma)=\frac 1 2\sum_{(i,j)\in E_\ep^1\setminus V_\ep^1}f_{\ep}^{\nn}(\bar\ffi_{\ep,\sigma}(j)-\bar\ffi_{\ep,\sigma}(i))\\
&&+\frac 1 2\sum_{(i,j)\in V_\ep^1}f_{\ep}^{\nn}(\bar\ffi_{\ep,\sigma}(j)-\bar\ffi_{\ep,\sigma}(i))
\\
&&=\frac 1 2\sum_{(i,j)\in E_\ep^1}f(\vth_{\ep,\sigma}(j)-\vth_{\ep,\sigma}(i))+\frac 1 2\sum_{(i,j)\in V_\ep^1}f_{\ep}^{\nn}(\bar\ffi_{\ep,\sigma}(j)-\bar\ffi_{\ep,\sigma}(i))
\\
&&\le \frac 1 2\sum_{(i,j)\in E_\ep^1\setminus V_\ep^1}f(\vth_{\ep,\sigma}(j)-\vth_{\ep,\sigma}(i))
+\frac 1 2\sum_{(i,j)\in V_\ep^1}f(\vth_{\ep,\sigma}(j)-\vth_{\ep,\sigma}(i))
+r(\ep,\sigma)
\\
&&=\F_\ep^\eff(\vth_{\ep,\sigma},B_\sigma)+r(\ep,\sigma)
\end{eqnarray*}
which proves \eqref{remain2} and concludes the proof of the lemma.
\end{proof}

\begin{remark}\label{transgrid}
For any $y\in\R^2$, set
\begin{equation}\label{truegamma1}
\gamma^y(\vep,\sigma):=\min_{\vph \in \mathcal{AF}_\vep(B_\sigma(y))} \left\{\F^{\nn}_\vep(\vph,B_\sigma(y)): e^{i\n\vph(\cdot)} = e^{i\theta(\cdot-y) }\text{ on } \partial_\vep B_\sigma(y)\right\};
\end{equation}
it is easy to check that for any $y\in\R^2$
$$
\gamma=\limsup_{\sigma\to 0}\limsup_{\ep\to 0} \gamma^y(\vep,\sigma)-\pi|{\textstyle \log\frac\ep\sigma}|=\liminf_{\sigma\to 0}\liminf_{\ep\to 0} \gamma^y(\vep,\sigma)-\pi|{\textstyle \log\frac\ep\sigma}|,
$$
where $\gamma$ is defined in \eqref{gamma}.
\end{remark}

Now we pass to a lemma which allows to deal with the far-field energy, i.e., with the energy outside suitable balls centered in the limit singularities. To this purpose,
for any $u\in\DD$, we denote,  as usual, $Ju^\n$ by  $\mu=\pi\sum_{i=1}^Md_i\delta_{x_i}$ (see \eqref{Dn} and \eqref{barDn}). Moreover, we recall that for any $D\subset\Omega$ open and for any $\rho>0$, $D_\rho(\mu)$ is defined according to \eqref{holedomain}.

\begin{lemma}\label{bulklemma}
Let $u\in\DD$ and
let $\Gamma_1,\ldots,\Gamma_M$ be pairwise disjoint segments such that $\Gamma_i$ joins $x_i$ with $\partial\Omega$ and $\mathcal H^1(\Gamma_i\cap S_u)=0$ for any $i=1,\ldots,M$. 
Fix $\rho>0$, and set $\Omega^\lin_\rho(\mu) := \Omega_\rho(\mu) \setminus \bigcup_{i = 1}^M\Gamma_i$. 
Then, there exist $w\in H^1(\Lin)$ with $[w] = \pm\frac{2\pi}{\n}$ on each $\Gamma_i$ and $\psi \in SBV(\Omega_\rho(\mu); \mathcal Z)$, such that $\ffi:=\psi+w$ is a lifting of $u$ in $\Omega_\rho(\mu)$, i.e., $e^{i\ffi}=u$.
In particular, for any $i=1,\ldots,M$
\begin{equation}\label{2pijumps}
([\psi]+[w])\res \Gamma_i\in 2\pi\Z.
\end{equation}

Moreover, there exist two sequences $\{w_h\}\subset C^\infty(\Lin)$ with $[w_h] = [w]$ and $\{\psi_h\}\subset SBV(\Omega_\rho(\mu);\mathcal Z)$ with polyhedral jump set $S_{\psi_h}$, such that, setting $\ffi_h:=\psi_h+w_h$,  the following properties are satisfied:
\begin{itemize}
\item[(i)] $w_h \to w$ in $H^1(\Lin)$;
\item[(ii)] $\psi_h \weakly \psi$ in $SBV(\Omega_\rho(\mu))$;
\item[(iii)] $\lim_{h\to\infty} \int_{S_{\ffi_h}} \phi( \ffi_h^+, \ffi_h^-, \nu_{\ffi_h}) \ud \mathcal H^1 =			\int_{S_\ffi} \phi(\ffi^+, \ffi^-,\nu_{\ffi}) \ud \mathcal H^1$ for any bounded and continuous integrand $\phi \colon \R^2\times \R^2\times  \mathcal S^1 \to \R^+$ satisfying $\phi(a , b,\nu) = \phi(b, a,-\nu)$ for all $( a, b,\nu) \in \R^2\times \R^2\times \mathcal S^1$. 
\end{itemize}

\end{lemma}

\begin{proof}
We will prove the lemma in five steps.

{\it Step 1: Construction of $w$ and $\psi$.} Let $\bar\ffi\in SBV^2_\loc(\Omega\setminus\bigcup_{i=1}^M\{x_i\})\cap SBV(\Omega)$ be a lifting of $u$. Such a lifting always exists in virtue of \cite[Remark 4]{DI}.
By \cite[Chapter 1]{BBH}, we have that $\nabla\bar\ffi$ is a conservative vector field on $\Lin$, and hence there exists a function $w\in H^1(\Lin)$ such that $\nabla w=\nabla\bar\ffi\res \Lin$. Moreover, using that $Ju^\n=\mu$ one can easily show that for any $i=1,\ldots,M$ 
\begin{equation}\label{jumpwffi}
[w]\res\Gamma_i=\pm\frac{2\pi}{\n}d_i,
\end{equation}
the sign depending on the orientation of the normal to $\Gamma_i$.
Since $\nabla(\bar\ffi-w)=0$ in $\Lin$, there exists a Caccioppoli partition $\{U_l\}_{l\in\N}$ of $\Omega_\rho(\mu)$ subordinated to $S_{\bar\ffi}\cup\bigcup_{i=1}^M\Gamma_i$, such that
\begin{equation}\label{H1Ai}
\bar\ffi-w=\sum_{l\in\N}c_{l}\chi_{U_{l}}\quad\textrm{ in }\Omega_\rho(\mu),
\end{equation}
for some $c_l\in\R$.
By \eqref{jumpwffi} and by the fact that $u^\n\in H^1(\Omega_\rho(\mu);\mathcal S^1)$, it follows that there exists a constant $c$ such that $c_l-c\in\frac{2\pi}{\n}\Z$  for any $l\in\N$.
Up to replacing $w$ with $w+c$, we can always assume that $c_l\in\frac{2\pi}{\n}\Z$.
For any $j=1,\ldots,\n-1$ we set
$$
E_{j}:=\bigcup_{l\in\N\,:\,c_l\in\frac{2\pi}{\n}j+2\pi\Z}U_l\,.
$$
Set moreover $\psi:=\sum_{j=1}^{\n-1}\frac{2\pi}{\n}j\chi_{E_j}$. By construction
\begin{equation}\label{limsupfinjump}
\mathcal H^1(S_\psi)\le \mathcal H^1(S_u)+\sum_{i=1}^M\mathcal H^1(\Gamma_i)<+\infty.
\end{equation}
Moreover, again by construction, $\psi+w-\bar\ffi\in2\pi\Z$; therefore, $\ffi:=\psi+w$ is a lifting of $u$,  and \eqref{2pijumps} is satisfied.

{\it Step 2: Approximation  of $w$.} 
By \eqref{jumpwffi}, the function $z:= e^{i\n w}$ belongs to $H^1(\Om_\rho(\mu);\mathcal{S}^1)$. Therefore, by \cite{SU}, there exists $\{z_h\}\subset C^\infty(\Omega_\rho(\mu);\mathcal S^1)$ such that $z_h\to z$ in $H^1(\Omega_\rho(\mu);\R^2)$. It is well known (see for instance \cite{B}) that for $h$ sufficiently large $\deg(z_h,\partial B_\rho(x_i))=d_i$ for any $i=1,\ldots,M$. Since $\Omega$ is simply connected, also $\Omega_\rho^\Gamma(\mu)$ is and hence there exists $\{\zeta_h\}\subset C^\infty(\Lin)$ such that $z_h=e^{i\zeta_h}$. Recalling that $\|z_h-z\|_{H^1(\Omega_\rho(\mu);\R^2)}\to 0$, we have that $\|z_h\bar z-1\|_{H^1(\Omega_\rho(\mu);\C)}\to 0$ (where we have identified $\R^2$ with  $\C$) and hence
$$
\|\nabla (\zeta_h-\n w)\|_{L^2(\Omega_\rho(\mu);\C)}=\|\nabla(z_h\bar z)\|_{L^2(\Omega_\rho(\mu);\C)}\to 0;
$$
this fact, combined with Poincar\'e inequality, yields
$$
\left\|\zeta_h-\fint_{\Lin}\zeta_h+\n\fint _{\Lin} w-\n w\right\|_{L^2(\Omega_\rho(\mu);\R^2)}\to 0.
$$
For any $h\in\N$ set
$$
w_h:=\frac{\zeta_h-\fint_{\Lin}\zeta_h} \n+\fint_{\Lin} w;
$$
then, $\{w_h\}\subset C^\infty(\Lin)$, $[w_h]=[w]$ and (i) is satisfied.

{\it Step 3: Approximation of $\psi$}. For every $i=1,\ldots,M$, we denote by $p_i$ the intersection point of $\Gamma_i$ with $\partial\Omega$ and by $n_i$ a (fixed) unit normal vector to $\Gamma_i$. 
Fix $\eta>0$.  For any $\xi>0$ and for any $i=1,\ldots,M$, we set
\begin{align*}
&\Rect_i(\xi,\eta):=\{\textstyle x_i+an_i+b\frac{p_i-x_i}{|p_i-x_i|}\,:\,-\xi\le a\le \xi,\,\rho+\eta\le b\le \mathcal H^1(\Gamma_i)-\eta\},\\
&L^\rig_i(\xi):=\{\textstyle x_i + \xi n_i+b\frac{p_i-x_i}{|p_i-x_i|}\,:\,\rho+\eta\le b\le \mathcal H^1(\Gamma_i)-\eta\},\\
& L^\lef_i(\xi):=\{\textstyle x_i - \xi n_i+b\frac{p_i-x_i}{|p_i-x_i|}:\,\rho+\eta\le b\le \mathcal H^1(\Gamma_i)-\eta\}.
\end{align*}

Moreover, for any $\delta>0$ we set
$$
\Omega^\Gamma_{\rho,\delta,\eta}(\mu):=\Omega^\Gamma_\rho(\mu)\setminus\bigcup_{i=1}^M \Rect_i(\delta,\eta).
$$

We notice that for  $\delta>0$ small enough,  $\Rect_i(2\delta,\eta)$ are pairwise disjoint and contained in $\Omega_\rho(\mu)$ and  $\partial\Omega^\Gamma_{\rho,\delta,\eta}(\mu)$ is Lipschitz continuous.
%
%
%
%
%
Then, by applying Theorem \ref{polyapprox} with  $A=\Omega^\Gamma_{\rho,\delta,\eta}(\mu)$, $g=\psi\res\Omega^\Gamma_{\rho,\delta,\eta}(\mu)$, and $\mathcal Z = \{\frac{2j \pi}{\n}\,:\,j = 1, \dots, \n-1\}$, we have that there exists a sequence $\{\psi_{h,\delta,\eta}\}\subset SBV(\Omega^\Gamma_{\rho,\delta,\eta}(\mu); \mathcal Z)$ such that, for any $h\in\N$, $S_{\psi_{h,\delta,\eta}}$ is polyhedral  and  the following properties are satisfied
\begin{equation}\label{strongl1}
\lim_{h\to\infty}\|\psi_{h,\delta,\eta} - \psi\|_{L^1(\Omega^\Gamma_{\rho,\delta,\eta}(\mu))}=0,
\end{equation}
\begin{equation}\label{weakconvmeas}
 D\psi_{h,\delta,\eta}\weakstar D(\psi\res \Omega^\Gamma_{\rho,\delta,\eta}(\mu))\qquad\textrm{as }h\to\infty, 
\end{equation}
and 
\begin{multline}
	\label{energyconvergence}
	\lim_{h\to\infty} \int_{S_{\psi_{h,\delta,\eta}}} \phi( \psi_{h,\delta,\eta}^+, \psi_{h,\delta,\eta}^-,\nu_{\psi_{h,\delta,\eta}}) \ud \mathcal H^1 \\
= \int_{S_\psi\cap \Omega^\Gamma_{\rho,\delta,\eta}(\mu)} \phi(\psi^+, \psi^-,\nu_\psi) \ud \mathcal H^1
\end{multline}
for all continuous $\phi \colon  \mathcal Z \times \mathcal Z\times \mathcal S^1  \to \R^+$ satisfying $\phi( a , b,\nu) = \phi( b, a,-\nu)$.

We  now extend $\psi_{h,\delta,\eta}$ from $\Omega^\Gamma_{\rho,\delta,\eta}(\mu)$ to $\Omega_\rho(\mu)$ by reflection along the segments $L_i^\rig(\delta)$ and $L_i^\lef(\delta)$ in the following manner (see also Figure \ref{reflection}).
\begin{figure}[h!]
	\centering
	\def\svgwidth{0.5\textwidth}
	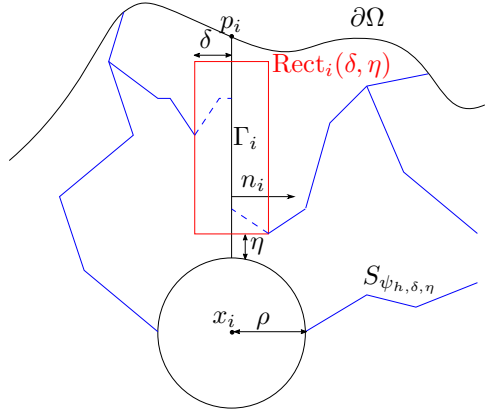
	\caption{Extension of $\psi_{h,\delta,\eta}$ into $\Omega_\rho(\mu)$ by reflection.}
\label{reflection}
\end{figure}
Set 
\begin{eqnarray*}
&\Rect^\rig_i(\delta,\eta):=\{\textstyle x_i+an_i+b\frac{p_i-x_i}{|p_i-x_i|}\,:\,0\le a\le \delta,\,\rho+\eta\le b\le \mathcal H^1(\Gamma_i)-\eta\},\\
&\Rect^\lef_i(\delta,\eta):=\{\textstyle x_i+an_i+b\frac{p_i-x_i}{|p_i-x_i|}\,:\,-\delta\le a\le 0,\,\rho+\eta\le b\le \mathcal H^1(\Gamma_i)-\eta\},
\end{eqnarray*}
and let $R_{i,\delta}^\rig(x):\Rect^\rig_{i}(\delta,\eta)\to \Rect_{i}(2\delta,\eta)\setminus \Rect_{i}(\delta,\eta)$ be defined by
\begin{equation}
R_{i,\delta}(x):=x+2\di(x, L^\rig_{i}(\delta))n_i.
\end{equation}
Analogously, we define the reflection map $R_{i,\delta}^\lef$.
We define the extensions $\bar \psi_{h,\delta,\eta}$ of $\psi_{h,\delta,\eta}$ in $\Omega_\rho(\mu)$ as 
$$
	 \bar\psi_{h,\delta,\eta}(x) := 
	\begin{cases}
		\psi_{h,\delta,\eta}(x) &\text{if } x \in \Omega^\Gamma_{\rho,\delta,\eta}(\mu), \\
		\psi_{h,\delta,\eta}(R_{i,\delta}^\rig(x)) &\text{if } x \in \Rect^\rig_i(\delta,\eta), \\
		\psi_{h,\delta,\eta}(R_{i,\delta}^\lef(x)) &\text{if } x \in \Rect^\lef_i(\delta,\eta).
	\end{cases}
$$ 

It is easy to see that $\bar\psi_{h,\delta,\eta}\in SBV(\Omega_\rho(\mu))$ and that $S_{ \bar\psi_{h,\delta,\eta}}$ is polyhedral. Let $\bar\psi^+_{h,\delta,\eta}$ (resp., $\bar\psi^-_{h,\delta,\eta}$) denote the trace of $\bar\psi_{h,\delta,\eta}$ on $\bigcup_{i=1}^M L^\rig_i(2\delta)$ (resp., $\bigcup_{i=1}^M L^\lef_i(2\delta)$ ).
By Theorem \ref{polyapprox} $|D\psi_{h,\delta,\rho}|\to |D(\psi\res\Omega_{\rho,\delta,\eta}^\Gamma(\mu))|$ and hence, (see for instance \cite[Theorem 2.11]{G}), we obtain
\begin{equation}\label{3i}
\lim_{h\to\infty}\|\bar\psi_{h,\delta,\eta}^+-\psi^+\|_{L^1(\bigcup_{i=1}^M L_i^\rig(2\delta))}=0,\quad \lim_{h\to\infty}\|\bar\psi_{h,\delta,\eta}^- -\psi^-\|_{L^1(\bigcup_{i=1}^M L_i^\lef(2\delta))}=0.
\end{equation}
Moreover, we also have that for any $i=1,\ldots,M$
\begin{equation}\label{3i1}
\lim_{\delta\to 0}\|\psi(\cdot\pm2\delta n_i)-\psi^\pm(\cdot)\|_{L^1(\Gamma_i\cap\Rect_i(\delta,\eta))}=0.
\end{equation}


Set $\ffi_{h,\delta,\eta}:=\bar\psi_{h,\delta,\eta}+w_h$.
By construction  and by \eqref{strongl1}, 
\begin{multline*}
	\limsup_{h \to \infty} \|\bar\psi_{h,\delta,\eta} - \psi\|_{L^1(\Omega_\rho(\mu))} \le\limsup_{h \to \infty} \|\bar\psi_{h,\delta,\eta} - \psi\|_{L^1(\cup_{i=1}^M \Rect_i(\delta,\eta))}\\
\leq \frac{2(\n-1)\pi}{\n} \sum_{i = 1}^M |\Rect_i(\delta,\eta)| \leq C\,\delta,
\end{multline*}
for some  constant $C>0$ depending on $M$, $\n$ and $\Omega$. Therefore 
\begin{equation}\label{l1conv}
\lim_{\eta\to 0}\lim_{\delta\to 0}\lim_{h\to\infty}\|\bar\psi_{h,\delta,\eta} - \psi\|_{L^1(\Omega_\rho(\mu))}=0.
\end{equation}

{\it Step 4: Convergence of the surface energy.}
In this step we prove that
\begin{eqnarray} \label{jumpenconv}
&&\lim_{\eta\to 0}\lim_{\delta\to 0}\lim_{h\to\infty}\int_{S_{\ffi_{h,\delta,\eta}}}\phi(\ffi_{h,\delta,\eta}^+,\ffi_{h,\delta,\eta}^-,\nu_{\ffi_{h,\delta,\eta}})\ud\mathcal H^1=\int_{S_{\ffi}}\phi(\ffi^+,\ffi^-,\nu_{\ffi})\ud\mathcal H^1.
\end{eqnarray}

Now we pass to the proof of \eqref{jumpenconv}.
Let $\phi$ be as in (iii).
Since by construction $w_h,\,w\in H^1(\Omega_\rho^\Gamma(\mu))$, in order to show \eqref{jumpenconv}, it is sufficient to prove
\begin{eqnarray}\label{jumpenconv1}
\lim_{\eta\to 0}\lim_{\delta\to 0}\lim_{h\to\infty}\int_{S_{\bar\psi_{h,\delta,\eta}}\setminus(\bigcup_{i=1}^M \Gamma_i)}\phi(\bar\psi_{h,\delta,\eta}^+,\bar\psi_{h,\delta,\eta}^-,\nu_{\bar\psi_{h,\delta,\eta}})\ud\mathcal H^1\\ \nonumber
=\int_{S_{\psi}\setminus(\bigcup_{i=1}^M\Gamma_i)}\phi(\psi^+,\psi^-,\nu_{\psi})\ud\mathcal H^1,
\\ \label{jumpenconv2}
\lim_{\eta\to 0}\lim_{\delta\to 0}\lim_{h\to\infty}\int_{S_{\ffi_{h,\delta,\eta}}\cap(\bigcup_{i=1}^M \Gamma_i)}\phi(\ffi_{h,\delta,\eta}^+,\ffi_{h,\delta,\eta}^-,\nu_{\ffi_{h,\delta,\eta}})\ud\mathcal H^1\\ \nonumber
=\int_{S_{\ffi}\cap(\bigcup_{i=1}^M \Gamma_i)}\phi(\ffi^+,\ffi^-,\nu_{\ffi})\ud\mathcal H^1.
\end{eqnarray}

We first show \eqref{jumpenconv1}. Set $\mathcal Z := \{\frac{2l\pi}{\n}\,:\, l=1, \dots, \n-1\}$. Since $\bar\psi_{h,\delta,\eta}$ takes values in $\mathcal Z$, using \eqref{energyconvergence}, we obtain
\begin{multline*}
\limsup_{h\to\infty}\left|\int_{S_{\bar\psi_{h,\delta,\eta}}\setminus(\bigcup_{i=1}^M \Rect_i(\delta,\eta))}\phi(\bar \psi_{h,\delta,\eta}^+, \bar\psi_{h,\delta,\eta}^-,\nu_{\bar \psi_{h,\delta,\eta}}) \ud \mathcal H^1\right.\\
\left.-\int_{S_{\psi}\setminus(\bigcup_{i=1}^M \Rect_i(\delta,\eta))}\phi(\psi^+, \psi^-,\nu_{ \psi}) \ud \mathcal H^1\right|=0.
\end{multline*}
By triangular inequality, it follows that
\begin{multline}\label{0}
\limsup_{h\to\infty}\left|\int_{(S_{\bar\psi_{h,\delta,\eta}}\setminus\bigcup_{i=1}^M\Gamma_i)\setminus(\bigcup_{i=1}^M \Rect_i(\delta,\eta))}\phi(\bar \psi_{h,\delta,\eta}^+, \bar\psi_{h,\delta,\eta}^-,\nu_{\bar \psi_{h,\delta,\eta}}) \ud \mathcal H^1\right.\\
-\left.\int_{(S_\psi\setminus\bigcup_{i=1}^M\Gamma_i)\setminus(\bigcup_{i=1}^M \Rect_i(\delta,\eta))}\phi(\psi^+, \psi^-,\nu_{\psi}) \ud \mathcal H^1\right|\\
\le \limsup_{h\to\infty}\int_{\bigcup_{i=1}^M(\Gamma_i\setminus\Rect_i(\delta,\eta))}(\phi(\bar \psi_{h,\delta,\eta}^+, \bar\psi_{h,\delta,\eta}^-,\nu_{\bar \psi_{h,\delta,\eta}}) +\phi(\psi^+, \psi^-,\nu_{ \psi})) \ud \mathcal H^1
\le r_0(\eta),
\end{multline}
with $\lim_{\eta\to 0}r_0(\eta)$=0.



For any $i=1,\ldots,M$, we set $I_{i}(\delta,\eta):=\partial \Rect_i(\delta,\eta)\setminus(L_i^\rig(\delta)\cup L_i^\lef(\delta))$
and we denote by $\Int_i(\delta,\eta)$ (resp., $\Int_i(2\delta,\eta)$)  the interior of $\Rect_i(\delta,\eta)\setminus \Gamma_i$ (resp., $\Rect_i(2\delta,\eta)\setminus \Gamma_i$ ). Fix $i=1,\ldots,M$.
Again by construction and by \eqref{energyconvergence}, we have that 
\begin{multline}\label{primo}
\limsup_{h\to\infty}\int_{S_{\bar\psi_{h,\delta,\eta}}\cap \Int_{i}(\delta,\eta)}\phi(\bar \psi_{h,\delta,\eta}^+, \bar\psi_{h,\delta,\eta}^-,\nu_{\bar \psi_{h,\delta,\eta}}) \ud \mathcal H^1\\
=\int_{S_\psi\cap(\Int_i(2\delta,\eta)\setminus \Int_i(\delta,\eta))} \phi(\psi^+, \psi^-,\nu_{\psi}) \ud \mathcal H^1\\
\le \int_{S_\psi\cap \Int_i(2\delta,\eta)} \phi(\psi^+, \psi^-,\nu_{\psi}) \ud \mathcal H^1=:r_1(\delta).
\end{multline}
with $\lim_{\delta\to 0}r_1(\delta)=0$.
Moreover, since by construction $[\bar\psi_{h,\delta,\eta}]=0$ on $L^\rig_i(\delta)\cup L^\lef_i(\delta)$, we have
\begin{multline}\label{secondo}
\limsup_{h\to\infty}\int_{S_{\bar\psi_{h,\delta,\eta}}\cap \partial\Rect_{i}(\delta,\eta)}(\phi(\bar \psi_{h,\delta,\eta}^+, \bar\psi_{h,\delta,\eta}^-,\nu_{\bar \psi_{h,\delta,\eta}}) +\phi(\psi^+,\psi^-,\nu_\psi))\ud \mathcal H^1\\
=\limsup_{h\to\infty}\int_{S_{\bar\psi_{h,\delta,\eta}}\cap I_{i}(\delta,\eta)}(\phi(\bar \psi_{h,\delta,\eta}^+, \bar\psi_{h,\delta,\eta}^-,\nu_{\bar \psi_{h,\delta,\eta}}) +\phi(\psi^+,\psi^-,\nu_\psi))\ud \mathcal H^1\\
\le C\delta=:r_2(\delta)
\end{multline}
with $\lim_{\delta\to 0}r_2(\delta)=0$.
By \eqref{primo} and \eqref{secondo}, and by triangular inequality, it follows that
\begin{multline}\label{1}
\limsup_{h\to\infty}\left|\int_{S_{\bar\psi_{h,\delta,\eta}}\cap (\Rect_i(\delta,\eta)\setminus\Gamma_i)}\phi(\bar \psi_{h,\delta,\eta}^+, \bar\psi_{h,\delta,\eta}^-,\nu_{\bar \psi_{h,\delta,\eta}}) \ud \mathcal H^1\right.\\
\left.-
\int_{S_\psi\cap(\Rect_i(\delta,\eta)\setminus\Gamma_i)} \phi(\psi^+, \psi^-,\nu_{\psi}) \ud \mathcal H^1\right|\le r_1(\delta)+r_2(\delta)\to 0\quad\textrm{as }\delta\to 0.
\end{multline}
Then \eqref{jumpenconv1} follows by summing \eqref{0} and \eqref{1}.

Finally we prove \eqref{jumpenconv2}. Fix $i=1,\ldots,M$. By \eqref{3i} and \eqref{3i1}, we have
$$
\lim_{\delta\to 0}\lim_{h\to\infty}\|\ffi_h^\pm-\ffi^\pm\|_{L^1(\Gamma_i\cap\Rect_i(\delta,\eta))}=0;
$$
then, by Dominated Convergence Theorem, we obtain
\begin{multline}\label{canchange}
\limsup_{h\to\infty}\left|\int_{\Gamma_i\cap\Rect_i(\delta,\eta)}\phi(\ffi_{h,\delta,\eta}^+, \ffi_{h,\delta,\eta}^-,\nu_{\ffi_{h,\delta,\eta}}) \ud\mathcal H^1\right.\\
\left.-\int_{\Gamma_i\cap\Rect_i(\delta,\eta)}\phi(\ffi^+, \ffi^-,\nu_{\ffi}) \ud\mathcal H^1\right|\le r(\delta),
\end{multline}
with $\lim_{\delta\to 0}r(\delta)=0$.
It follows that
\begin{multline}\label{finished}
\limsup_{h\to\infty}\left|\int_{\Gamma_i}\phi(\ffi_{h,\delta,\eta}^+, \ffi_{h,\delta,\eta}^-,\nu_{\ffi_{h,\delta,\eta}}) \ud\mathcal H^1-\int_{\Gamma_i }\phi(\ffi^+, \ffi^-,\nu_{\ffi}) \ud\mathcal H^1\right|\\
\le \limsup_{h\to\infty}\left|\int_{\Gamma_i\cap\Rect_i(\delta,\eta)}\phi(\ffi_{h,\delta,\eta}^+, \ffi_{h,\delta,\eta}^-,\nu_{\ffi_{h,\delta,\eta}}) -\phi(\ffi^+, \ffi^-,\nu_{\ffi}) \ud\mathcal H^1\right|\\
+ \limsup_{h\to\infty}\int_{\Gamma_i\setminus\Rect_i(\delta,\eta)}(\phi(\ffi_{h,\delta,\eta}^+, \ffi_{h,\delta,\eta}^-,\nu_{\ffi_{h,\delta,\eta}}) +\phi(\ffi^+, \ffi^-,\nu_{\ffi}) )\ud\mathcal H^1\\
\le r(\delta)+4\|\phi\|_{L^\infty}\eta.
\end{multline}
By summing \eqref{finished} over $i=1,\ldots,M$, we obtain \eqref{jumpenconv2}.
This concludes the proof of \eqref{jumpenconv}.

{\it Step 5: Conclusion}
Using \eqref{l1conv} and \eqref{jumpenconv}, we show that (ii) and (iii) hold true.
By a standard diagonal argument in $h$, $\delta$ and $\eta$, 
there  exists  a sequence of functions $\psi_{h}:=\bar\psi_{h,\delta(h),\eta(h)}\in SBV(\Omega_\rho(\mu))$ with polyhedral jump set, such that $\ffi_h:=\psi_h+w_h$ satisfies (iii). Moreover (ii) follows by Theorem \ref{SBVComp}, up to extracting a further subsequence.
\end{proof}

\begin{remark}
We notice that  in Lemma \ref{bulklemma}(iii) the  assumption that $\phi$ is bounded can be replaced by requiring that
$\phi(a,b,\nu)=\Theta(|a-b|)\Phi(\nu)$, for some continuous functions $\Theta$ and $\Phi$. In fact, the former assumption is used only in the proof of \eqref{canchange}.
\end{remark}

Now we are in a position to prove the $\Gamma$-limsup inequality.
We recall that for any $u\in\D_M^\n$, $Ju^\n=\pi\sum_{i=1}^Md_i\delta_{x_i}$, with $|d_i|=1$ and $x_i\in\Omega$.  To ease the notations, for any $i=1,\ldots,M$, we set $\theta_i(\cdot):=\theta(\cdot-x_i)$, where $\theta$ is the angular polar coordinate (with respect to the origin).

\begin{proof}[Proof of Theorem \ref{mainthm}(iii)]
Without loss of generality we can assume that $\W(u^\n)<+\infty$ and hence there exists a constant $C>0$ such that for any $\sigma>0$
$$
\frac 1 2\int_{\Omega_\sigma(\mu)}|\nabla u^\n|^2\ud x\le M\pi|\log\sigma|+C.
$$
Fix $\sigma>0$ and let $C_{i,k}$ denote the annulus $B_{2^{-k}\sigma}(x_i)\setminus B_{2^{-k-1}\sigma}(x_i)$ for any $k\in\N$ and for any $i=1,\ldots,M$. 
 Arguing as in the proof of \cite[Theorem 4.5]{ADGP} (see also \cite{AP}), one can show that for any $k\in\N$ there exists a unitary vector $a_{i,k}$  such that
\begin{align} \nonumber
&\left\|u^\n-a_{i,k}\left(\frac{x-x_i}{|x-x_i|}\right)^{d_i}\right\|_{H^1(C_{i,k};\R^2)}=:r_1(i,k)\to 0\quad{ as}\,k\to\infty,\\ \label{annulest}
&\left\|\nabla(\n\ffi-d_i\theta_{i})\right\|_{L^2(C_{i,k};\R^2)}=:r_2(i,k)\to 0\quad{ as}\,k\to\infty,
\end{align}
for any lifting $\ffi$ of $u$ with $\ffi\in SBV(\Omega)\cap SBV^2_\loc(\Omega\setminus \bigcup_{i=1}^M\{x_i\})$ (a lifting like that exists by \cite[Remark 4]{DI}).
Fix $k\in\N$ and let $\{\psi_{k,h}\}_h$ and $\{w_{k,h}\}_h$ be the sequences constructed in Lemma \ref{bulklemma} for $\rho=2^{-k-2}\sigma$. 
For any $h\in\N$, set $\ffi_{k,h}:=\psi_{k,h}+w_{k,h}$. Let moreover $\ffi_k=\psi_k+w_k$ be the lifting of $u$ constructed in Lemma \ref{bulklemma} with $\rho=2^{-k-2}\sigma$.
By \eqref{annulest} and by Lemma \ref{bulklemma}(i), using also that $\nabla \ffi_{k}(x)=\nabla w_{k}(x)$ for a.e. $x\in \Omega_{2^{-k-2}\sigma}(\mu)$, it follows that 
\begin{multline}\label{annulest2}
\frac 1 2\int_{C_{i,k}}\left|\nabla \left(w_{k,h}-d_i\frac{\theta_i} \n\right)\right|^2\ud x
\le
\int_{C_{i,k}}\left|\nabla w_{k,h}-\nabla w_k\right|^2\ud x\\
+\int_{C_{i,k}}\left|\nabla \left(\ffi_{k}-d_i\frac{\theta_i} \n\right)\right|^2\ud x 
=r(i,k,h),
\end{multline}
with $\lim_{k\to\infty}\lim_{h\to \infty}r(i,k,h)=0$.
By construction and recalling that $[w_{k,h}]=[w_k]=\frac{2\pi}{\n}$, we have that $\nabla(w_{k,h}-\sum_{i=1}^M d_i\frac{\theta_i}{\n})$ is a conservative field in $\Omega_{2^{-k-2}\sigma}(\mu)$, and so for any $i=1,\ldots, M$ there exists a function $\xi_{i,k,h}\in C^\infty(B_{2^{-k}\sigma}(x_i)\setminus B_{2^{-k-2}\sigma}(x_i))$ with null average such that
$$
\nabla\xi_{i,k,h}=\nabla\left(w_{k,h}-d_i\frac{\theta_i}{\n}\right)\quad \textrm{in }B_{2^{-k}\sigma}(x_i)\setminus B_{2^{-k-2}\sigma}(x_i).
$$
By Poincar\'e inequality and by \eqref{annulest2}, we have
\begin{equation}\label{xiest}
\|\xi_{i,k,h}\|_{H^1(B_{2^{-k}\sigma}(x_i)\setminus B_{2^{-k-2}\sigma}(x_i))}\le C\,r(i,k,h),
\end{equation} 
for some $C>0$ depending only on $\sigma$ (and not on $k$).
Let $\eta\in C^{1}\left([\frac 1 2,1];[0,1]\right)$ be such that $\eta(t)=1$ for $t\in[\frac 1 2,\frac 5 8]$ and $\eta(t)=0$ for $t\in[\frac 7 8, 1] $, and set
\begin{equation}\label{cutoff}
\bar\ffi_{i,k,h}(x):=\ffi_{k,h}(x)-\eta(2^{k}\sigma^{-1}|x-x_i|)\xi_{i,k,h}(x) \quad\textrm{ for any }x\in C_{i,k}.
\end{equation}
Set $g_{i,k,h}(y):=\bar\ffi_{i,k,h}(x_i+2^{-k}\sigma y)$.
By \eqref{xiest} and \eqref{cutoff}, it easily follows that
\begin{equation}\label{dausare}
\lim_{h\to\infty}\|\tilde g_{i,k,h}(y)-a_{i,k}(\textstyle \frac {y}{|y|})^{d_i}\|_{H^1(B_1\setminus B_{\frac 1 2}; \R^2)}=0.
\end{equation}

We define 
\begin{equation}\label{newffi}
\bar\ffi_{k,h}(x):=\left\{\begin{array}{ll}
\ffi_{k,h}(x)&\textrm{ if }x\in\Omega_{2^{-k}\sigma}(\mu)\\
\bar\ffi_{i,k,h}(x)&\textrm{ if }x\in C_{i,k}\\
\ffi_{k,h}(x)-\xi_{i,k,h}(x)&\textrm{ if }x\in B_{2^{-k-1}\sigma}(x_i)\setminus B_{2^{-k-2}\sigma}(x_i).
\end{array}\right. 
\end{equation}
Since $\nabla\ffi_{k,h}=\nabla w_{k,h}$, we have that for any $i=1,\ldots,M$
$$
\nabla\bar\ffi_{k,h}=\frac{d_i}{\n}\nabla{\theta_i}\qquad\textrm{ in } B_{2^{-k-1}\sigma}(x_i)\setminus B_{2^{-k-2}\sigma}(x_i),
$$
whence
$$
e^{i\n\bar\ffi_{k,h}}=e^{id_i\theta_i}\quad\textrm{ in } B_{2^{-k-1}\sigma}(x_i)\setminus B_{2^{-k-2}\sigma}(x_i).
$$
Moreover, for any $i=1,\ldots,M$, let $\bar\ffi_{i,k,h,\ep}\in\mathcal{AF}_\ep(B_{2^{-k-1}\sigma}(x_i))$ be a solution of the minimum problem $\gamma^{x_i}(\vep, 2^{-k-1}\sigma)$ with $\bar\ffi_{i,k,h,\ep}=\bar\ffi_{i,k,h}$ on $\partial_\ep B_{2^{-k-1}\sigma}(x_i)$.
By a standard density argument we can always assume that $\Omega_\ep^0\cap S_{\ffi_{k,h}}=\emptyset$, so we can set
$$
\ffi_{k,h,\ep}:=\left\{\begin{array}{ll}
\bar\ffi_{k,h}&\textrm{in }(\Omega_{2^{-k-1}\sigma}(\mu))_{\ep}^0\\
\bar\ffi_{i,k,h,\ep}&\textrm{in }(B_{2^{-k-1}\sigma}(x_i))_{\ep}^0.
\end{array}
\right.
$$
Set $u_{k,h,\ep}:=e^{i\ffi_{k,h,\ep}}$ and let $\hat{u}_{k,h,\ep}$ denote the piecewise affine interpolation of $u_{k,h,\ep}$ constructed according with \eqref{paiofu}. In order to conclude the proof it is enough to show 
\begin{align}\label{uconv}
&\limsup_{k\to\infty}\limsup_{h\to\infty}\limsup_{\ep\to 0}\|u-\hat u_{k,h,\ep}\|_{L^2(\Omega_\ep;\R^2)}=0,\\ \label{enconv}
&\limsup_{k\to\infty}\limsup_{h\to\infty}\limsup_{\ep\to 0}\F_\ep^{\nn}(\ffi_{k,h,\ep})-M\pi|\log\ep|=\F^{\nn}_0(u)\,;
\end{align}
indeed, by \eqref{uconv},\eqref{enconv} and Theorem \ref{SBVComp}, using a standard diagonal argument, there exists  a sequence $\{\ffi_{\ep}\}:=\{\ffi_{k(\ep), h(\ep),\ep}\}$  satisfying \eqref{gammalimsup} and with $\hat u_\ep\weakly u$ in $SBV^2_\loc(\Omega\setminus\cup_{i=1}^M\{x_i\};\R^2)$.

We now prove \eqref{uconv} and \eqref{enconv}.
To this purpose set
$$
T_{k,h,\ep}:=\bigcup_{\newatop {i\in  (\Omega_{2^{-k-1}\sigma}(\mu))_{\ep}^2}{(i+\ep Q)\cap S_{\bar\ffi_{k,h}}\neq\emptyset}}(i+\ep Q).
$$
By standard interpolation estimates (see for instance \cite{C}), we have
\begin{equation}
\limsup_{\ep\to 0}\|e^{i\ffi_{k,h}}-\hat u_{k,h,\ep}\|_{L^2((\Omega_{2^{-k}\sigma}(\mu))_\ep \setminus T_{k,h,\ep};\R^2)}=0,
\end{equation}
which combined with Lemma \ref{bulklemma} yields
\begin{equation}\label{almostclaim}
\limsup_{k\to\infty}\limsup_{h\to\infty}\limsup_{\ep\to 0}\|u-\hat u_{k,h,\ep}\|_{L^2((\Omega_{2^{-k}\sigma}(\mu))_\ep \setminus T_{k,h,\ep};\R^2)}=0.
\end{equation}
Moreover, by Lemma \ref{bulklemma}, (for $h$ sufficiently large) 
\begin{equation}\label{saltolimitato}
\mathcal H^1(S_{\bar\ffi_{k,h}})\le C\quad\textrm{ so that }\quad\lim_{\ep\to 0}|T_{k,h,\ep}|=0.
\end{equation}
Since $\|u-\hat u_{k,h,\ep}\|_{L^\infty}\le 2$, by \eqref{saltolimitato} 
\begin{multline}\label{negligclaim}
\|u-\hat u_{k,h,\ep}\|_{L^2(\cup_{i=1}^M B_{2^{-k+1}\sigma}(x_i);\R^2)}+\|u-\hat u_{k,h,\ep}\|_{L^2(T_{k,h,\ep};\R^2)}\\
\le 2(M|B_{2^{-k+1}\sigma}|+|T_{k,h,\ep}|)\to 0\qquad\textrm{ as }\ep\to 0,\,k\to\infty.
\end{multline}
Therefore, \eqref{uconv} immediately follows by \eqref{almostclaim} and \eqref{negligclaim}.

Now we pass to the proof of \eqref{enconv}.
We set
\begin{align*}
& E_{k,\ep}:=\{(i,j)\in\Omega_{\ep}^1\,:\,[i,j]\cap\cup_{i=1}^M \partial C_{i,k}\neq\emptyset\},
\end{align*}
and
\begin{align*}
N_{k,h,\ep}&:=\frac 1 2\sum_{(i,j)\in E_{k,\ep}\setminus (T_{k,h,\ep})_\ep^1}f_{\ep}^{\nn}(\ffi_{k,h,\ep}(j)-\ffi_{k,h,\ep}(i)),\\
N'_{k,h,\vep} & := \sum_{i = 1}^M \left( \F_\ep^{\nn}\left(\ffi_{k,h,\ep},C_{i,k}\setminus T_{k,h,\ep}\right) - \pi \log 2 \right), \\
I_{k,h,\vep}& := \sum_{i = 1}^M \left(\F_\ep^{\nn}\left(\ffi_{k,h,\ep}, B_{2^{-k-1}\sigma}(x_i)\right)-\pi |\log\textstyle  \frac \vep {2^{-k-1}\sigma}|\right), \\
	I'_{k,h,\vep} & := \F_\ep^{\nn}\left(\ffi_{k,h,\ep}, \Omega_{2^{-k}\sigma}(\mu) \setminus T_{k,h,\vep} \right) - M \pi |\log (2^{-k}\sigma)|, \\
	I''_{k,h,\vep} & :=\frac 1 2\sum_{(i,j)\in (T_{k,h,\ep})_\ep^1}f_{\ep}^{\nn}(\ffi_{k,h,\ep}(j)-\ffi_{k,h,\ep}(i)),
\end{align*}
We notice that
\begin{equation}\label{splitting}
	\F_\ep^{\nn}(\ffi_{k,h,\ep})-M\pi|\log\ep| = N_{k,h,\vep} + N'_{k,h,\vep} + I_{k,h,\vep} + I'_{k,h,\vep}+I''_{k,h,\ep}.
\end{equation}
We investigate the convergence of the terms $N_{k,h,\ep},\,N'_{k,h,\ep},\,I_{k,h,\ep},\,I'_{k,h,\ep},\,I''_{k,h,\ep}$  separately. 
By construction and by standard interpolation estimates, we have
\begin{equation}\label{firstt}
\limsup_{k\to \infty}\limsup_{h\to \infty}\limsup_{\ep\to 0} N_{k,h,\ep}=0.
\end{equation}
By \eqref{dausare} and again by standard interpolation estimates
\begin{equation}\label{thirdt0}
\limsup_{k\to\infty}\limsup_{h\to\infty}\limsup_{\ep\to 0}N'_{k,h,\ep}=0.
\end{equation}
By construction,
\begin{equation}\label{secondt}
\limsup_{k\to\infty}\limsup_{h\to\infty}\limsup_{\ep\to 0}I_{k,h,\ep}=M\gamma.
\end{equation}
By standard interpolation estimates, \eqref{saltolimitato}, and Lemma \ref{bulklemma}, we have 
\begin{equation}\label{thirdt}
\limsup_{k\to\infty}\limsup_{h\to\infty}\limsup_{\ep\to 0}I'_{k,h,\ep}=\W(u^\n).
\end{equation}
Now we prove  that 
\begin{equation}\label{lasterm}
\lim_{k\to\infty}\lim_{h \to \infty}\lim_{\ep\to 0} I''_{k,h,\vep}=\int_{S_u}|\nu_u|_1\ud\mathcal H^1.
\end{equation}
By construction,  for any $(i,j)\in (T_{k,h,\ep})_\ep^1$ we have
$$
\di(\ffi_{k,h,\ep}(j)-\ffi_{k,h,\ep}(i),\frac{2\pi}{\n}\Z)\le \ep\left\|\nabla \bar \ffi_{k,h}\right\|_{L^\infty(\Omega_{2^{-k-1}\sigma}(\mu);\R^2)}\le C_{k,h}\ep,
$$
where the last inequality follows by the very definition of $\bar\ffi_{k,h}$ in \eqref{newffi} (see also \eqref{cutoff}).
Therefore, by the very definition of $f_\ep^{\nn}$ (see \eqref{potential}), for $\ep$ sufficiently small there holds
\begin{align}\label{contosuibond}
&f_\ep^{\nn}(\ffi_{k,h,\ep}(j)-\ffi_{k,h,\ep}(i))\le C_{k,h}^2\ep^2\textrm{ if }\di(\ffi_{k,h,\ep}(j)-\ffi_{k,h,\ep}(i),{2\pi}\Z)\le C_{k,h}\ep,\\ \nonumber
&f_\ep^{\nn}(\ffi_{k,h,\ep}(j)-\ffi_{k,h,\ep}(i))=\ep\textrm{ if }\di(\ffi_{k,h,\ep}(j)-\ffi_{k,h,\ep}(i),\frac{2\pi}{\n}(\Z\setminus\n\Z))\le C_{k,h}\ep.
\end{align}
Let $\phi\in C(\R\times\R\times\mathcal S^1)$ be defined by
$$
\phi(a,b,\nu):=\left\{\begin{array}{ll}
|\sin(\n(a-b))|\,|\nu|_1&\textrm{if }\di(a-b,2\pi\Z)\le\frac{\pi}{2\n}\\
|\nu|_1&\textrm{otherwise}.
\end{array}\right.
$$
Notice that given any segment $I\subset\R^2$ with normal $\nu_I$, we have that $\ep$ times the number of the bonds in $(\R^2)_\ep^1$  intersecting $I$ converges, as $\ep\to 0$, to $\mathcal H^1(I)|\nu_I|_1$.
This fact, together with \eqref{contosuibond}, yields
\begin{equation}\label{almend}
\lim_{\ep\to 0}I''_{k,h,\vep}=\int_{S_{\bar\ffi_{k,h}}\cap\Omega_{2^{-k-1}\sigma}(\mu)}\phi(\bar\ffi_{k,h}^+, \bar\ffi_{k,h}^-,\nu_{\bar\ffi_{k,h}})\ud\mathcal H^1.
\end{equation}
Therefore,  by using Lemma \ref{bulklemma} (iii) and the fact that $[\bar\ffi_{k,h}]=[\ffi_{k,h}]$ (see \eqref{cutoff} and \eqref{newffi}), we obtain
\begin{multline}
	\label{Dconv}
	\lim_{k\to\infty}\lim_{h \to \infty}\lim_{\ep\to 0} I''_{k,h,\vep} \\
=\lim_{k\to\infty}\lim_{h \to \infty}\int_{S_{\bar\ffi_{k,h}}\cap\Omega_{2^{-k-1}\sigma}(\mu)}\phi(\bar\ffi_{k,h}^+, \bar\ffi_{k,h}^-,\nu_{\bar\ffi_{k,h}})\ud\mathcal H^1 \\
=\lim_{k\to\infty}\int_{S_{\ffi_{k}}\cap\Omega_{2^{-k-1}\sigma}(\mu)}\phi(\ffi_{k}^+, \ffi_{k}^-,\nu_{\ffi_{k}})\ud\mathcal H^1\\
=\lim_{k\to\infty}\int_{S_u\cap \Omega_{2^{-k-1}\sigma}(\mu)}|\nu_u|_1\ud\mathcal H^1=\int_{S_u}|\nu_u|_1\ud\mathcal H^1,
\end{multline}
which is exactly \eqref{lasterm}.

Finally, by\eqref{splitting} and by summing \eqref{firstt}, \eqref{secondt}, \eqref{thirdt0} \eqref{thirdt} and \eqref{lasterm}, we get \eqref{enconv}, which concludes the proof. 
\end{proof}
\vskip4mm
\textbf{Acknowledgements.}  
The work of RB, MC, LDL was supported by the DFG Collaborative Research Center TRR 109, ``Discretization in
Geometry and Dynamics''. Part of this work was carried out while MP was visiting TUM thanks to the Visiting Professors Program
funded by the Bavarian State Ministry for Science, Research and the Arts.

\end{document}